\documentclass[11pt]{article}
\usepackage{amsmath,amssymb,amsthm}
\usepackage{natbib}
\usepackage{booktabs}

\theoremstyle{plain}
\newtheorem{theorem}{Theorem}[section]
\newtheorem{corollary}[theorem]{Corollary}
\newtheorem{lemma}[theorem]{Lemma}
\newtheorem{proposition}[theorem]{Proposition}

\theoremstyle{definition}

\newtheorem{definition}[theorem]{Definition}
\newtheorem{problem}[theorem]{Problem}

\theoremstyle{remark}
\newtheorem{remark}[theorem]{Remark}

\usepackage{graphicx}

\usepackage{xcolor,hyperref}
\definecolor{darkblue}{rgb}{0.0,0.0,0.6}
\hypersetup{colorlinks,breaklinks,linkcolor=darkblue,urlcolor=darkblue,anchorcolor=darkblue,citecolor=darkblue}

\begin{document}

\title{Certified High-Dimensional Wasserstein \\ Robust Portfolio Optimization}
\author{%
Chung-Han Hsieh\thanks{Corresponding author. Email: \texttt{ch.hsieh@mx.nthu.edu.tw}}
\quad
Rong Gan\thanks{Email: \texttt{GanZong@gmail.com}}\\[0.5em]
\small Department of Quantitative Finance, National Tsing Hua University\\
\small Hsinchu 30004, Taiwan}
\date{}

\maketitle

\begin{abstract}
We develop a certified, scalable approximation for high-dimensional Wasserstein distributionally robust portfolio optimization. For expected-utility maximization under order-one Wasserstein ambiguity, standard duality yields a semi-infinite convex program. For long-only portfolios with box support under the one-norm ground metric, an exact sample-specific vertex reformulation provides an exponential-size computational benchmark. We then majorize the utility by supporting hyperplanes and dualize the support subproblems, obtaining a finite hyperplane--dual formulation over compact polyhedral supports. Under the one-norm ground metric and polyhedral portfolio constraints, this formulation is a polynomial-size linear program. The uniform utility-approximation error bounds both the robust-value error and the near-optimality gap for the original robust problem. Experiments validate the certified approximation and demonstrate monthly 476-asset rebalancing and computational scalability to 1,000 assets.
\end{abstract}

    
 


\section{Introduction} 
\label{section: Introduction}
    Data-driven portfolio optimization must be re-solved repeatedly as new return observations become available. A portfolio decision $w\in\mathcal W$ is ideally chosen to maximize expected utility under the unknown return distribution~$\mathbb F$. In practice, however, only finitely many historical observations~$\{\widehat X_j\}_{j=1}^N$ are available, and sample average approximation (SAA) replaces~$\mathbb F$ with their empirical distribution~$\widehat{\mathbb F}$. Although SAA is asymptotically consistent under standard conditions~\cite{shapiro2021lectures}, its finite-sample optimizer can amplify estimation error and perform poorly out of sample, a longstanding concern in portfolio optimization~\cite{michaud1989markowitz,gao2023finite}.

    Distributionally robust optimization (DRO) addresses this sampling uncertainty by optimizing against the worst-case distribution in a data-centered ambiguity set~\cite{wiesemann2014distributionally,rahimian2022frameworks,kuhn2025distributionally}. We consider an order-$1$ Wasserstein ball centered at~$\widehat{\mathbb F}$. Besides supporting finite-sample guarantees and asymptotic consistency~\cite{mohajerin2018data,zhao2018data}, its transport geometry allows probability mass to move from each observation to arbitrary points in the return support, rather than merely reweighting the observed~scenarios.

    This modeling flexibility creates the central computational difficulty studied in this paper. Standard Wasserstein duality yields a semi-infinite expected-utility formulation whose constraints range over the entire return support. Our objective is to develop a scalable finite approximation with explicit guarantees relative to the original robust~objective.

\subsection{Related Work and Computational Gap}

    Several studies have explored DRO under the Wasserstein metric, laying the foundation for both theoretical developments and practical applications.
    Early foundational work established finite-sample guarantees, asymptotic consistency, and finite convex reformulations for Wasserstein-based DRO~\cite{mohajerin2018data}, while risk-averse two-stage Wasserstein DRO models were studied in~\cite{zhao2018data}, including worst-case distributions and convergence results.
    Mean-variance Wasserstein portfolio optimization and data-driven radius selection were studied in~\cite{blanchet2022distributionally}.
    Kelly-based portfolio formulations under Wasserstein ambiguity were developed in~\cite{li2023wasserstein} and later extended to settings with convex transaction costs in~\cite{hsieh2026flight}.
    A duality derivation based on the interchangeability principle was developed in~\cite{zhang2024short}.
    Most closely related to our approximation architecture, \citet{hsieh2024solving} develops a supporting-hyperplane reformulation for log-utility portfolios under polyhedral ambiguity over finitely many prescribed return scenarios.

Within the computational DRO literature, \cite{ji2020data} develop reformulation and finitely convergent bisection methods for distributionally robust reward--risk ratios, whereas~\cite{cheramin2022computationally} construct dimension-reduced inner and outer approximations under moment and Wasserstein ambiguity with optimality-gap guarantees. These methods address fractional objectives or high-dimensional ambiguity representations.

Here, Wasserstein ambiguity allows probability mass to move from each observed sample to arbitrary locations in the compact polyhedral return support, rather than merely reweighting a fixed scenario set. For concave expected-utility maximization, the resulting Wasserstein dual remains a support-indexed semi-infinite convex program. Under long-only box support with an~$\ell_1$ ground norm, an exact finite reduction is possible, but its explicit constraint representation can grow exponentially with the portfolio dimension. This leaves a gap between exactness and scalable implementation: a finite high-dimensional formulation with explicit guarantees relative to the original robust objective.

\subsection{Computational Contributions}

    We address this gap for an order-$1$ Wasserstein ambiguity set induced by a fixed $\ell_p$ ground norm and a compact polyhedral return support. The proposed hyperplane--dual formulation is a finite norm-constrained convex program. For $p=1$ and polyhedral portfolio constraints, it becomes a single linear program whose numbers of variables and constraints grow polynomially with the problem dimensions.

    Specifically, this paper makes four contributions. First, building on the standard Wasserstein dual reformulation, we develop a finite hyperplane--dual approximation for concave increasing utilities over compact polyhedral return supports. The construction combines supporting-hyperplane majorization of the utility, the dual-norm representation of the order-$1$ Wasserstein penalty, and LP duality for the support minimization.

    Second, we establish a uniform approximation certificate: the utility approximation error directly bounds both the robust-value error and the near-optimality gap for the original Wasserstein DRO~problem.

    Third, for long-only portfolios with box return support and an $\ell_1$ ground norm, we derive an exact sample--vertex reformulation of the original DRO problem. This formulation provides the benchmark used to measure the approximation fidelity of the proposed method.

    Fourth, for box return support, we establish a finite-threshold large-ambiguity characterization: once the Wasserstein radius exceeds the support diameter, the original and surrogate DRO problems share a closed-form maximin optimizer~set.

\paragraph{Notation}
For $x,y\in\mathbb R^n$, we write $\langle x,y\rangle:=x^\top y$, and all vector inequalities are understood componentwise. Fix $p\in[1,\infty]$, let $q$ be its H\"older conjugate, and write $\|\cdot\|=\|\cdot\|_p$ for the ground norm and $\|\cdot\|_*=\|\cdot\|_q$ for its dual, where $1/p+1/q=1$ with $1/\infty=0$. The symbols $\mathbf 0$ and $\mathbf 1$ denote the zero and all-ones vectors of the appropriate dimension, respectively, and $e_i$ denotes the $i$th standard basis vector.

\section{Preliminaries} 
\label{Section: Preliminaries} 
    This section introduces the necessary preliminaries for formulating the DRO~problem.

\subsection{Support and Empirical Distribution}
\label{section: data}
   The decision-making problem is driven by a random vector~$X \in \mathbb{R}^n$ with an unknown distribution~$\mathbb{F}$, supported on the polyhedral set
    \begin{align} \label{eq: compact polyhedron}
        \mathfrak{X}
            := \{x \in \mathbb{R}^n: Hx \leq h\}.
    \end{align}
    Here $H \in \mathbb{R}^{r_{\mathfrak{X}}\times n}$ and $h \in \mathbb{R}^{r_{\mathfrak{X}}}$ are chosen so that $\mathfrak{X}$ is nonempty and bounded;\footnote{
        For example, the typical box-type support $\mathfrak{X}_{\rm box}:=\{x \in \mathbb{R}^n : x_{\min} \leq x \leq x_{\max} \}$ is a special case of~\eqref{eq: compact polyhedron} when 
        $
        H :=
            \begin{bmatrix}
            I\\
            -I
            \end{bmatrix} \in \mathbb{R}^{2n \times n},
        $
        and
        $
        h :=
            \begin{bmatrix}
            x_{\max}\\
            -x_{\min}
            \end{bmatrix} \in\mathbb R^{2n}.
        $
    } $r_{\mathfrak{X}}$ denotes the number of linear inequalities used to describe the support.
    The decision maker observes $N$ independent and identically distributed (i.i.d.) random samples, denoted by $\widehat{X}_1, \ldots , \widehat{X}_N \in \mathbb{R}^n$, drawn from the true distribution~$\mathbb{F}$.
    Based on these samples, the empirical distribution, denoted by $\widehat{\mathbb{F}}$, is defined as~$
    \widehat{\mathbb{F}} := \frac{1}{N} \sum_{j=1}^N \delta_{\widehat{X}_j},
    $ 
    where~$\delta_{\widehat{X}_j}$ is the Dirac-delta~measure\footnote{
        $\delta_{\widehat{X}_j}(A) = 1$ if~$\widehat{X}_j \in A$ and $\delta_{\widehat{X}_j}(A) = 0$ otherwise, for any measurable set $A \subseteq \mathfrak{X}$.
        }
    at $\widehat{X}_j$ for~$j = 1, \cdots, N$.

\subsection{Wasserstein Ambiguity Set}
    The \emph{Wasserstein metric} measures the distance between two probability distributions defined on a given support space; see, e.g.,~\cite{mohajerin2018data,kuhn2025distributionally,hsieh2026flight}. Let $\mathcal{M}(\mathfrak{X})$ be the set of all probability distributions supported on $\mathfrak{X} \subseteq \mathbb{R}^n$.

\begin{definition}[$1$-Wasserstein Metric with $\ell_p$ Ground Norm] 
    For any two distributions $\mathbb{F}_1, \mathbb{F}_2 \in \mathcal{M}(\mathfrak{X})$, the \emph{$1$-Wasserstein metric with respect to the $\ell_p$ ground norm}, denoted by~$d_p(\mathbb{F}_1, \mathbb{F}_2) : \mathcal{M}(\mathfrak{X}) \times \mathcal{M}(\mathfrak{X}) \rightarrow \mathbb{R}_+$, is defined as:
    \begin{align*}
        d_p(\mathbb{F}_1, \mathbb{F}_2) 
        &:= 
        \inf_{\Pi \in \mathcal{C}(\mathbb{F}_1, \mathbb{F}_2)}  
        \mathbb{E}^{(X_1, X_2) \sim \Pi} \left[ \| X _1- X_2 \|_p \right]
    \end{align*}
    where $X_1 \sim \mathbb{F}_1$ and $X_2 \sim \mathbb{F}_2$, and $\mathcal{C}(\mathbb{F}_1, \mathbb{F}_2)$ denotes the set of joint distributions $\Pi$ with marginals $\mathbb{F}_1$ and $\mathbb{F}_2$.
\end{definition}

\begin{definition}[$1$-Wasserstein Ambiguity Set] 
    Given a radius~$\varepsilon > 0$, the \emph{$1$-Wasserstein ambiguity set induced by the $\ell_p$ ground norm} centered at~$\widehat{\mathbb{F}}$, denoted by~$\mathcal{B}_{\varepsilon} (\widehat{\mathbb{F}})$, is defined~as
    $
        \mathcal{B}_{\varepsilon} ( \widehat{\mathbb{F}} ) 
            := \big\{ \mathbb{F} \in \mathcal{M}(\mathfrak{X}) : d_p(\mathbb{F}, \widehat{\mathbb{F}}) \leq \varepsilon \big\}.
    $
\end{definition}

\subsection{Distributionally Robust Optimization Problem}
    Observing only empirical data $\widehat{\mathbb{F}}$, the decision maker adopts a DRO approach, selecting a decision vector $w$ from a nonempty compact polyhedral feasible set~$\mathcal{W} \subseteq \mathbb{R}^n$ that maximizes the worst-case expected utility across all distributions in the Wasserstein ambiguity set $\mathcal{B}_{\varepsilon} (\widehat{\mathbb{F}})$. Let $U(\cdot)$ be a concave, strictly increasing, and continuously differentiable utility function. In the portfolio problem considered here, utility depends on the decision and return only through the scalar portfolio return~$\langle w,X\rangle$. 

\begin{problem}[Distributionally Robust Optimization Problem]
    Given a radius $\varepsilon > 0$, the \emph{distributionally robust optimization} (DRO) problem is formulated as:
    \begin{align}\label{eq: DRO problem}
        \sup_{w \in \mathcal{W}} \,
        \inf_{\mathbb{F} \in \mathcal{B}_{\varepsilon}(\widehat{\mathbb{F}})}
        \mathbb{E}^{\mathbb{F}} \left[ 
            U( \langle w, X \rangle )\right].
    \end{align}
\end{problem}

\begin{remark}[Extreme Cases of the Wasserstein DRO]
    When $\varepsilon = 0$, the Wasserstein ambiguity set reduces to an empirical distribution, i.e., 
    $\mathcal{B}_{\varepsilon}(\widehat{\mathbb{F}}) = \{ \widehat{\mathbb{F}} \}$, and the DRO problem~\eqref{eq: DRO problem} coincides with the sample average approximation (SAA).
    When~$\varepsilon \uparrow \infty$, the ambiguity set $\mathcal{B}_{\varepsilon}(\widehat{\mathbb{F}}) \uparrow \mathcal{M}(\mathfrak{X})$.
    In this case, the DRO problem reduces to a robust optimization (RO) formulation; see \cite{zhao2018data}.
\end{remark}

\section{Main Results}
\label{Section: Theoretical Reformulation}
This section first derives a semi-infinite convex reformulation
of~\eqref{eq: DRO problem}. Under long-only box support with an
$\ell_1$ ground norm, we also obtain an exact sample--vertex
reformulation that serves as a computational benchmark. 
Our main result constructs a supporting-hyperplane surrogate and derives its finite hyperplane--dual formulation over compact polyhedral return supports.
When $p=1$ and the portfolio constraints are polyhedral, the resulting
formulation is a polynomial-size linear program. We then show that the
uniform utility-approximation error bounds both the robust-value error
and the near-optimality gap for the original DRO problem. Finally, for
box support, we establish the finite-threshold large-ambiguity
characterization.

\begin{lemma}[Semi-Infinite Convex Reformulation]
    \label{lemma: A Convex Reformulation of the DRO Problem} 
    Given observed sample vectors $\widehat{x}_1, \cdots, \widehat{x}_N$, for any $\varepsilon > 0$, the DRO problem~\eqref{eq: DRO problem} admits an equivalent semi-infinite convex reformulation: 
    \begin{align}\label{eq: convex DRO}
        &\displaystyle
        \sup_{\substack{w\in\mathcal{W},\ \lambda\ge0\\ a_j\in\mathbb R,\ j=1,\dots,N}}
        -\lambda\varepsilon + \frac{1}{N} \sum_{j=1}^{N} a_j\\
        &{\rm s.t.} \; 
           U \left( \langle w, x \rangle \right) + \lambda  \|x-\widehat{x}_j\|  \geq a_j,\quad j = 1, \dots, N, \; x \in \mathfrak{X} \notag
    \end{align}
\end{lemma}

\begin{proof}
    The result follows the duality argument in \cite[Theorem 4.2]{mohajerin2018data}.  
\end{proof}

\begin{remark}
  The reformulation in Lemma~\ref{lemma: A Convex Reformulation of the DRO Problem} is convex but semi-infinite. Proposition~\ref{prop: exact sample vertex reformulation} shows that, under the long-only, box-support, and $\ell_1$ structure, it admits an exact finite reduction.
\end{remark}

\begin{proposition}[Exact Sample-Vertex Reformulation]
\label{prop: exact sample vertex reformulation}
    Suppose that the ground norm is $\ell_1$, the support is the box
    $\mathfrak{X}:=\{x\in\mathbb R^n:x_{\min}\leq x\leq x_{\max}\}$ with $
    x_{\min},x_{\max}\in\mathbb R^n$, and
    $\mathcal{W}\subseteq\mathbb R_+^n$. For each sample
    $\widehat{x}_j\in\mathfrak{X}$, define a sample-specific vertex set
    $
    \mathcal V_j
        :=
        \left\{
            v\in\mathbb R^n:
            v_i\in\{(x_{\min})_i, \ (\widehat{x}_j)_i\},
            \ i=1,\dots,n
        \right\}.
    $
    Then the original DRO problem~\eqref{eq: DRO problem} is equivalent to
    \begin{align}
       & \sup_{\substack{w\in\mathcal{W},\ \lambda\geq0\\a_j\in\mathbb R}}
        -\lambda\varepsilon+\frac1N\sum_{j=1}^N a_j
        \label{eq: exact sample vertex DRO}\\
       & \mathrm{s.t.}\;
        U(\langle w,v\rangle)
        +\lambda\|v-\widehat{x}_j\|_1
        \geq a_j,
        \quad v\in\mathcal V_j,\ j=1,\dots,N. \notag
    \end{align}
    Hence, the original DRO admits an exact finite convex reformulation
    with at most $N2^n$ sample--vertex constraints.
\end{proposition}
\begin{proof}
    See Appendix~\ref{appendix: technical proofs}.
\end{proof}

\begin{remark}
    Proposition~\ref{prop: exact sample vertex reformulation}
    provides an exact finite reduction, but its explicit representation contains up to $N2^n$ sample--vertex constraints. Thus, a direct implementation based on full vertex enumeration has exponential worst-case size in~$n$. The supporting-hyperplane construction developed below instead yields a single polynomial-size finite formulation. Section~\ref{sec: Computational Scalability} describes an exact constraint-generation benchmark that avoids full vertex~enumeration.
\end{remark}

\subsection{Supporting Hyperplane Approximation} 
\label{subsection: supporting hyperplane approximation}
    To obtain a formulation whose size grows polynomially with the problem dimensions, we adopt the supporting-hyperplane approximation developed in \cite{hsieh2024solving}. This construction applies to compact polyhedral supports and a fixed $\ell_p$ ground norm, and leads to the finite norm-constrained convex formulation developed below. For $p=1$, the resulting formulation is a linear~program.

The supporting hyperplane approximation approach begins by identifying~$M$ partition points within the range of inputs for the utility function~$U(\langle w, x\rangle)$. Subsequently,~$M$ hyperplanes are constructed as a surrogate function to approximate the utility function.

For any $w\in\mathcal{W}$ and $x\in\mathfrak{X}$, let $y := \langle w, x \rangle$.
Since $\mathcal{W}$ and~$\mathfrak{X}$ are compact,~$y$ also lies in the compact~interval
\[
\underline{y} := \min_{w\in\mathcal{W},\; x \in \mathfrak{X}} \langle w,x\rangle,
\qquad
\overline{y} := \max_{w\in\mathcal{W},\; x \in \mathfrak{X}} \langle w,x\rangle .
\]
Define an auxiliary function $f:[\underline{y}, \overline{y}] \rightarrow \mathbb{R}$ by $f(y):=U(y)$.
Assume that~$f$ is continuously differentiable, strictly increasing, and concave.
According to~\cite{hsieh2024solving}, the supporting hyperplane method approximates $f$ by partitioning the interval $[\underline{y},\overline{y}]$ and constructing the corresponding tangent hyperplanes.

Specifically, for an integer $M \geq 2$, choose the partition points $\{ y_m\}_{m=1}^{M}$ such that
$
    \underline{y} = y_1 < y_2 < \dots < y_M = \overline{y}.
$
For each $m=1, \dots, M$, define the affine function
$
    h_m (y) := \alpha_m y + \beta_m , 
$
where $\alpha_m$ and $\beta_m$ are the slope and the intercept coefficients, respectively, satisfying
\begin{align*}
    \alpha_m := f'(y_m) 
    \; \text{ and } \;
    \beta_m := f(y_m) - \alpha_m y_m.
\end{align*}
Equivalently,
$
    h_m(y)=f(y_m)+f'(y_m)(y-y_m).
$
Hence, $h_m$ is tangent to $f$ at $y_m$ and satisfies $h_m(y_m) = f(y_m)$ for each $m = 1, \dots, M$. Moreover, by the concavity of~$f$, we have 
$
f(y)\leq h_m(y)$ for $y \in [\underline{y}, \overline{y}]$ and $
m=1,\dots,M.$
Consequently,
\[
    f(y)\leq \min_{1\leq m \leq M} h_m(y),
    \qquad
    y \in [\underline{y}, \overline{y}].
\]
Thus, the utility function admits the upper approximation 
\begin{align} \label{eq: supporting hyperplane approximation to the utility}
    U( \langle w, x \rangle )  
    = 
    f( \langle w, x \rangle ) 
    \approx 
    \min_{1 \leq m \leq M} \left[ \alpha_m \langle w ,x \rangle + \beta_m \right]. 
\end{align}

\paragraph{Approximation Error}  
The approximation error can be characterized uniformly on $[\underline{y}, \overline{y}]$.
Suppose, in addition, that $f'$ is Lipschitz continuous on this interval with Lipschitz constant~$L_f$.
Define
\begin{align} \label{eq: hyperplane surrogate}
    \widehat{U}_M(y):=\min_{1\leq m\leq M}(\alpha_m y+\beta_m),
\end{align}
and 
\[
    \Delta_y := \max_{m=1,\dots,M-1}(y_{m+1}-y_m).
\]
Applying the tangent-majorant estimate in~\cite[Lemma~3.1]{hsieh2024accelerating} on each subinterval $[y_m, y_{m+1}] $ yields
\[
    0 \leq \widehat{U}_M(y) - f(y) \leq \frac{L_f}{8} \Delta_y^2,
    \qquad
    \forall y \in  [\underline{y},\overline{y}].
\]
Thus, for any tolerance $\eta > 0$, choosing $\Delta_y\leq \sqrt{ \frac{8\eta}{L_f}}$ guarantees
\begin{align} \label{ineq: eta-uniform bound}
    0 \leq \widehat{U}_M(y)-f(y) \leq \eta,
    \qquad
    \forall y\in[\underline{y},\overline{y}].
\end{align}
Hence, we define
{\small \[
    \begin{aligned}
    M_\eta^*
    :=
    \min\Big\{ M \geq 2 : \; &
    \exists\ \underline{y}=y_1<\cdots<y_M=\overline{y} \quad
    \text{s.t. }
    0\leq \widehat{U}_M(y)-f(y) \leq \eta,\ 
    \forall y\in[\underline{y},\overline{y}]
    \Big\}.
    \end{aligned}
\]
}The mesh condition
$
    \Delta_y\leq \sqrt{8\eta/L_f}
$
is a simple sufficient condition for the selected hyperplanes to achieve uniform error at most~$\eta$. 

\begin{problem}[Supporting-Hyperplane DRO Surrogate] 
    \label{problem: supporting hyperplane DRO surrogate}
    For a prescribed tolerance $\eta > 0$, let~$\widehat{U}_{M_\eta^*}$ be the supporting-hyperplane surrogate with uniform error at most $\eta$ constructed above. We refer~to
    \[
    \widehat{V}_\eta^\star(\varepsilon)
    :=
    \sup_{w \in \mathcal{W}} \,
    \inf_{\mathbb{F}\in\mathcal{B}_\varepsilon(\widehat{\mathbb{F}})}
    \mathbb{E}^{\mathbb{F}} \left[ \widehat{U}_{M_\eta^*}(\langle w,X\rangle) \right]
    \]
    as the surrogate Wasserstein DRO problem.
\end{problem}

\subsection{Finite Convex Approximation via Supporting Hyperplanes}
    We now apply the hyperplane approximation framework to the Wasserstein DRO problem. For a fixed $\ell_p$ ground norm, this yields a finite norm-constrained convex program; in the computationally central case $p=1$, the formulation reduces to a robust linear program.

\begin{theorem}[Finite Hyperplane-Dual Reformulation]
\label{theorem: DRO via hyperplane}
    Given a hyperplane approximation tolerance~$\eta > 0$ and the corresponding optimal number of hyperplanes $M_\eta^*$, the surrogate Wasserstein DRO Problem~\ref{problem: supporting hyperplane DRO surrogate} admits the following finite hyperplane-dual reformulation:
    \begin{align}
    & \sup_{\substack{
    w\in\mathcal{W},\ \lambda \geq 0,\\ a_j, \; 
    \mu_j^m \geq 0\; \forall j,m
    }}
    -\lambda\varepsilon+\frac{1}{N}\sum_{j=1}^N a_j \label{eq: DRO hyperplane lemma}\\
    {\rm s.t.}\quad
    &\alpha_m \langle w, \widehat{x}_j\rangle+\beta_m
        +\langle \mu_j^m,H\widehat{x}_j-h\rangle
        \geq a_j,
        \qquad j\leq N,\ m \leq M_\eta^*; \notag\\
    &\|-\alpha_m w-H^\top\mu_j^m\|_*
    \le \lambda,\qquad j\le N,\ m \leq M_\eta^*. \notag
\end{align} 
    where $\|\cdot\|_*=\|\cdot\|_q$ is the dual norm of the chosen $\ell_p$ ground norm, and $\mu_j^m\in\mathbb R^{r_{\mathfrak{X}}}_+$ is the support-dual variable associated with the sample--hyperplane pair $(j,m)$ with $j = 1, \dots, N$ and $m = 1, \dots, M_\eta^*$.
\end{theorem}
\begin{proof}
    See Appendix~\ref{appendix: technical proofs}.
\end{proof}

\begin{corollary}[Box-Specialized Hyperplane-Dual Reformulation]
\label{corollary: box-specialized HYP reformulation}
Suppose that the ground norm is $\ell_1$, the support is the box
$
    \mathfrak{X}
    = \{x\in\mathbb R^n:x_{\min}\leq x\leq x_{\max}\},
$
and $\mathcal W\subseteq\mathbb R_+^n$.
Then the surrogate Wasserstein DRO problem in Theorem~\ref{theorem: DRO via hyperplane} reduces to
\begin{align}
    &\sup_{\substack{
        w\in\mathcal W,\ \lambda\geq0,\\
        a_j\in\mathbb R,\;
        s^m\in\mathbb R_+^n
    }}
    -\lambda\varepsilon+\frac1N\sum_{j=1}^N a_j
    \label{eq: box-specialized HYP reformulation}\\
    \mathrm{s.t.}\quad
    &\alpha_m\langle w,\widehat x_j\rangle+\beta_m
    -\langle\widehat x_j-x_{\min},s^m\rangle
    \geq a_j,
    \quad j=1,\dots,N,\quad m=1,\dots,M_\eta^*,
    \notag\\
    &s^m\geq\alpha_m w-\lambda\mathbf 1,
    \quad m=1,\dots,M_\eta^*.
    \notag
\end{align}
If $\mathcal W$ is polyhedral, this is a linear program with $n+1+N+nM_\eta^*$ scalar decision variables.
\end{corollary}
\begin{proof}
    See Appendix~\ref{appendix: technical proofs}.
\end{proof}

\begin{remark}[Formulation Size and Approximation Accuracy]
\label{remark: LP size}
    Let $M:=M_\eta^*$ and suppose that the feasible set $\mathcal{W}$ is described by $r_{\mathcal{W}}$ affine constraints. For a polyhedral support~\eqref{eq: compact polyhedron}, the finite formulation in Theorem~\ref{theorem: DRO via hyperplane} has
    $
        n+1+N+r_{\mathfrak{X}}NM
    $
    scalar decision variables: $n$ portfolio weights, one Wasserstein dual variable $\lambda$, $N$ epigraph variables $a_j$, and $r_{\mathfrak{X}}NM$ support-dual variables~$\mu_j^m$.

    For the computationally central case $p=1$, the dual norm is $\ell_\infty$, so each constraint
    $
        \|-\alpha_m w-H^\top\mu_j^m\|_\infty\leq \lambda
    $   
    is equivalent to $2n$ scalar linear inequalities. Hence, excluding the constraints defining $\mathcal{W}$, the formulation contains~$NM$ robust-value constraints, $2nNM$ norm-dual inequalities, and~$r_{\mathfrak{X}}NM$ nonnegativity constraints on the support-dual variables.

Consequently, for fixed support description size $r_{\mathfrak{X}}$, the formulation grows polynomially, and in fact linearly in each of $n$, $N$, and $M_\eta^*$ when the others are fixed. 
For a uniform partition with mesh size $\Delta_y$, the number of hyperplanes satisfies $
M = O\!\left(\frac{\overline{y}-\underline{y}}{\Delta_y}\right).
$
Since the sufficient condition~$\Delta_y \leq \sqrt{8\eta/L_f}$, developed in Section~\ref{subsection: supporting hyperplane approximation}, guarantees approximation error at most~$\eta$, we have
$
    M_\eta^*
    =
    O\!\left((\overline{y} - \underline{y}) \sqrt{L_f/\eta}\right).
$
Hence, for fixed support range and curvature constant, the $p=1$ robust linear formulation has size
\[
    O\!\left((n+r_{\mathfrak{X}})N M_\eta^* \right)
    =
    O\!\left((n+r_{\mathfrak{X}})N\eta^{-1/2}\right).
\]
For the box support, $r_{\mathfrak{X}}=2n$, and this reduces to
$
    O\!\left((n+r_{\mathfrak{X}})N\eta^{-1/2}\right)
    = O(nN\eta^{-1/2}).
$
By contrast, under the long-only conditions
of Proposition~\ref{prop: exact sample vertex reformulation}, the exact
sample-vertex formulation contains
$
\sum_{j=1}^{N}|\mathcal V_j|
\leq N2^n
$ robust-value constraints. 
A direct implementation that enumerates these sample-specific vertices therefore has exponential worst-case size in $n$, whereas the certified hyperplane-dual formulation has polynomial size and avoids explicit sample-vertex enumeration.
\end{remark}

\begin{remark}[Computational Class of the Ground Norm]
\label{remark: ground norm extension}
    Theorem~\ref{theorem: DRO via hyperplane} indicates that the computational class depends on $p$. For $p=1$, $q=\infty$, and the dual-norm constraint is equivalent to finitely many linear inequalities; thus, when $\mathcal{W}$ is polyhedral, the formulation is a robust linear program. 
    The same is true for $p=\infty$, since $q=1$ and the constraint
    $
        \|-\alpha_m w - H^\top\mu_j^m\|_1 \leq \lambda
    $
    admits an exact linear epigraph representation. More generally, whenever the dual-norm unit ball is polyhedral, the finite reformulation remains a linear program. For $p=2$, the formulation is a second-order cone program, while for general $1<p<\infty$ it is a finite convex norm-constrained program.
\end{remark}

\begin{proposition}[Certified Robust Value and Near-Optimality]
\label{prop: robust value approximation}
    Fix $\eta>0$ and let $\widehat{U}_{M_\eta^*}$ be the supporting-hyperplane surrogate with uniform error at most~$\eta$, defined in~\eqref{eq: hyperplane surrogate}, satisfying
    \[
        0\leq \widehat{U}_{M_\eta^*}(y)-U(y) \leq \eta,
        \qquad
        \forall y\in[\underline{y},\overline{y}].
    \]
    Define the original and surrogate Wasserstein robust values by
    \[
    V^\star(\varepsilon)
    :=
    \sup_{w\in\mathcal{W}} 
    \inf_{\mathbb{F}\in\mathcal{B}_\varepsilon(\widehat{\mathbb{F}})}
    \mathbb E^{\mathbb{F}} \left[ U(\langle w,X\rangle)\right],
    \]
    and
    \[
    \widehat V_\eta^\star(\varepsilon)
    :=
    \sup_{w\in\mathcal{W}}
    \inf_{\mathbb{F}\in\mathcal{B}_\varepsilon(\widehat{\mathbb{F}})}
    \mathbb E^{\mathbb{F}} \left[ \widehat{U}_{M_\eta^*}(\langle w,X\rangle) \right].
    \]
    Then
    $
    V^\star(\varepsilon)
    \leq \widehat V_\eta^\star(\varepsilon)
    \leq V^\star(\varepsilon)+\eta.
    $
    Moreover, if $\widehat{w}_\eta$ is an optimizer of the surrogate problem, then
    \[
    \inf_{\mathbb{F}\in\mathcal{B}_\varepsilon(\widehat{\mathbb{F}})}
    \mathbb E^{\mathbb{F}}
    \left[ U(\langle \widehat w_\eta, X\rangle) \right]
    \geq
        V^\star(\varepsilon)-\eta.
    \]
\end{proposition}
\begin{proof}
    See Appendix~\ref{appendix: technical proofs}.
\end{proof}

\begin{corollary}[Out-of-Sample Reliability Certificate]
\label{corollary: out-of-sample reliability}
    Let $\widehat{\mathbb F}_N$ be the empirical distribution of $N$
    i.i.d.\ observations from $\mathbb F$, and let $\mathbb{P}^N$ denote the corresponding $N$-fold sampling law. 
    Let $\widehat w_{\eta,N}(\varepsilon)$ and $\widehat V_{\eta,N}^{\star}(\varepsilon)$ denote an optimizer and the optimal value, respectively, of the surrogate problem based on~$\widehat{\mathbb F}_N$. Define
    $
        g(w):=\mathbb E^{\mathbb F}[U(\langle w,X\rangle)].
    $
    On the event
    $
        d_p(\mathbb F,\widehat{\mathbb F}_N)\leq\varepsilon,
    $
    we have
    \[
        g\bigl(\widehat w_{\eta,N}(\varepsilon)\bigr)
        \geq
        \widehat V_{\eta,N}^{\star}(\varepsilon)-\eta.
    \]
    Consequently,
    $
        \mathbb{P}^N\!\left\{
        g\bigl(\widehat w_{\eta,N}(\varepsilon)\bigr)
        \geq
        \widehat V_{\eta,N}^{\star}(\varepsilon)-\eta
        \right\}
        \geq
        \mathbb{P}^N\!\left\{
        d_p(\mathbb F,\widehat{\mathbb F}_N)\leq\varepsilon
        \right\}.
    $
\end{corollary}
\begin{proof}
    See Appendix~\ref{appendix: technical proofs}.
\end{proof}

Section~\ref{sec: Multiple-Period Case} evaluates the proposed
formulation in a rolling setting, where the empirical distribution is
updated and the finite program in Theorem~\ref{theorem: DRO via hyperplane} is resolved at each
rebalancing date.

\begin{remark}[Relation to Existing Tractable Reformulations]
    \cite{mohajerin2018data} derive finite
    reformulations for several structured loss classes, including
    piecewise-affine losses over polyhedral supports. Building on the
    supporting-hyperplane construction in~\cite{hsieh2024solving}, we
    approximate a general concave increasing utility and derive the
    corresponding finite hyperplane--dual formulation. Beyond tractability, we show that the uniform utility-approximation error directly bounds the robust-value error and the near-optimality gap for the original DRO~problem.
\end{remark}

\subsection{Large-Ambiguity Optimality Characterization}
    This subsection characterizes optimal portfolios once the Wasserstein radius is sufficiently large. Throughout this subsection, we specialize the return support to the box
    \begin{align} \label{eq: compact box}
        \mathfrak{X}_{\rm box} =\{x\in\mathbb R^n:x_{\min} \leq x \leq x_{\max}\},
    \end{align}
    and take the feasible set to be the unit simplex,
        \[
        \mathcal{W}:=\Delta_n
        =
        \{w\in\mathbb R^n:w^\top\mathbf 1=1,\; w\geq \mathbf 0\}.
        \]
    For this support, define the support diameter
    $
        \bar{\varepsilon}:=\sup_{x,x'\in\mathfrak{X}}\|x-x'\|.
    $
    These specializations are used only in the large-ambiguity characterization, where they yield a closed-form optimizer set. We then show that the same finite-threshold characterization applies to the hyperplane-approximation formulation in Theorem~\ref{theorem: DRO via hyperplane}.

\begin{theorem}[Large-Ambiguity Optimality]
\label{theorem: large ambiguity optimality}
    Given a compact support set $\mathfrak{X}= \{x \in \mathbb{R}^n : x_{\min} \leq x \leq x_{\max}\}$
    and feasible set $\mathcal{W}=\Delta_n$, define 
	\[
	c^*:=\max_{1\leq i\leq n} x_{\min,i},
	\quad
	\mathcal I^*:=\{i\in\{1,\dots,n\}:x_{\min,i}=c^*\}.
	\]
	Then, for every $\varepsilon\geq\bar\varepsilon$, the DRO problem~\eqref{eq: DRO problem} has the same optimal solution set as
    $
    \max_{w\in\Delta_n} \langle w,x_{\min}\rangle,
    $
    namely
    $
    W^* 
    = \operatorname{conv}\{e_i \in \mathbb{R}^n : i\in\mathcal I^*\}.
    $
    In particular, if $\mathcal I^*=\{i^*\}$, then $W^*=\{e_{i^*}\}$.
    If $x_{\min}=c\mathbf 1$ for some $c\in\mathbb R$, then
    $\mathcal I^*=\{1,\dots,n\}$ and hence $W^*=\Delta_n$.
\end{theorem}
\begin{proof}
    See Appendix~\ref{appendix: technical proofs}.
\end{proof}

\begin{remark}[Relation to the Asymptotic Large-Ambiguity Limit]
    The finite-threshold statement in Theorem~\ref{theorem: large ambiguity optimality}
    is stronger than the usual asymptotic characterization. Since
    $\mathcal{B}_\varepsilon(\widehat{\mathbb{F}})=\mathcal M(\mathfrak{X})$ for every
    $\varepsilon\geq\bar\varepsilon$, the same optimal solution set is obtained in
    the limit $\varepsilon\to\infty$.
\end{remark}

\begin{corollary}[Large-Ambiguity Optimality of the Hyperplane Approximation]
    \label{corollary: Large-Ambiguity Optimality of the Hyperplane Approximation}
    Consider the supporting-hyperplane DRO surrogate problem~\ref{problem: supporting hyperplane DRO surrogate}. Suppose that $\mathcal{W}=\Delta_n$ and that $\alpha_m>0$ for all
    $m=1,\dots, M_\eta^*$. Define
    \[
    c^* := \max_{1\leq i\leq n} x_{\min,i},
    \quad
    \mathcal{I}^*:=\{i\in\{1,\dots,n\}:x_{\min,i}=c^*\}.
    \]
    Then, for every $\varepsilon \geq \bar{\varepsilon}$, the hyperplane-approximation formulation has the same optimal solution set as the original DRO problem, namely
    \[
    W^*
        = \{w\in\Delta_n:w_i=0\ \text{for all }i\notin\mathcal I^*\}
        = \operatorname{conv}\{e_i:i\in\mathcal I^*\}.
    \]
    In particular, if $\mathcal I^*=\{i^*\}$, then $W^*=\{e_{i^*}\}$.
    If $x_{\min}=c\mathbf 1$ for some $c\in\mathbb R$, then $W^*=\Delta_n$.
\end{corollary}
\begin{proof}
    See Appendix~\ref{appendix: technical proofs}.
\end{proof}

\section{Empirical Studies}
\label{Section: Empirical Studies}
  
    We evaluate three formulations. The proposed hyperplane--dual formulation in Theorem~\ref{theorem: DRO via hyperplane} is denoted by $\texttt{DRO}_{\texttt{HYP}}$. The exact sample-specific vertex reformulation in Proposition~\ref{prop: exact sample vertex reformulation} is denoted by $\texttt{DRO}_{\texttt{EX}}$. An \texttt{RSOME} implementation of the same hyperplane-approximated DRO problem is denoted by $\texttt{RSOME}_{\texttt{HYP}}$. We use $\texttt{DRO}_{\texttt{EX}}$ to assess approximation fidelity to the original DRO problem and $\texttt{RSOME}_{\texttt{HYP}}$ to assess implementation consistency. The $\texttt{DRO}_{\texttt{HYP}}$ and $\texttt{DRO}_{\texttt{EX}}$ formulations are implemented in \texttt{CVXPY}~\cite{diamond2016cvxpy}, whereas $\texttt{RSOME}_{\texttt{HYP}}$ is implemented in \texttt{RSOME}~\cite{chen2020robust}; all three formulations are solved with MOSEK. All experiments in this section are conducted on a laptop equipped with a 3.2~GHz processor and 16~GB of~RAM.

\paragraph{Portfolio Setting}
    We consider a portfolio with $n \geq 1$ assets whose returns are represented by the random vector $X \in \mathbb{R}^n$, supported on the compact box~\eqref{eq: compact box}, i.e.,
    $
    \mathfrak{X} = \mathfrak{X}_{\rm box}.
    $
    All empirical DRO models use the $\ell_1$ ground norm. In the
    market-data experiments, the box bounds are set separately for each estimation window to the sample-wide minimum and maximum observed returns and are applied uniformly across~assets.   
    
    We assume $x_{\min} > -\mathbf{1}$ to ensure strictly positive prices.
    The true distribution $\mathbb{F}$ of returns $X$ is unknown; instead, the investor observes $N$ return samples $\{\widehat{X}_j\}_{j=1}^N$.
    The feasible set for the portfolio-weight vector is defined as~$
    \mathcal{W} 
    := 
    \left\{ 
        w \in \mathbb{R}^n : {w}^\top \mathbf{1} = 1, \; {w} \geq \mathbf{0} 
    \right\}
    $, representing a \emph{long-only} and \emph{cash-financed} strategy.

\paragraph{Logarithmic Growth Utility}
    We adopt the \emph{logarithmic growth rate} of wealth as the utility function, i.e.,  
    $
        U( \langle w, x \rangle ) :=  \log(1 + \langle w, x\rangle), 
    $
    which is a common choice in sequential investment problems; see, e.g., \cite{Cover_Thomas_2012,Luenberger_2013_investment_sci,Rujeerapaiboon2014robust, hsieh2024solving, hsieh2026flight}.
    The investor's goal is to choose a portfolio weight $w \in \mathcal{W}$ so as to maximize the expected logarithmic growth rate of wealth under the worst-case return distribution within the ambiguity set:
    $$
        \max_{w \in \mathcal{W}} \, 
        \inf_{\mathbb{F} \in \mathcal{B}_{\varepsilon}(\widehat{\mathbb{F}})} 
        \mathbb{E}^{\mathbb{F}}  \left[  \log(1 + \langle w, X\rangle) \right].
    $$
    The empirical studies proceed in four stages.
First, Section~\ref{sec: Single-Period Cases} assesses approximation
fidelity, implementation consistency, and computation time against the
exact and \texttt{RSOME} benchmarks.
Second, Section~\ref{sec: Multiple-Period Case} evaluates repeated
deployment in an S\&P~500-scale rolling portfolio.
Third, Section~\ref{sec: Computational Scalability} isolates the effect
of portfolio dimension under a controlled experimental design.
Finally, Section~\ref{sec: Impact of Wasserstein Radius} examines how
the Wasserstein radius affects out-of-sample performance and reliability.

\subsection{Single-Period Benchmarks} 
\label{sec: Single-Period Cases}
    This subsection compares $\texttt{DRO}_{\texttt{HYP}}$ with $\texttt{DRO}_{\texttt{EX}}$ and $\texttt{RSOME}_{\texttt{HYP}}$ on an 11-asset single-period portfolio. We evaluate solution fidelity through the resulting portfolio weights and computational efficiency through the average solution time over the Wasserstein-radius grid.

\paragraph{Data and Experimental Setup}
    We consider an 11-asset portfolio comprising~10 risky assets and one risk-free asset. The risky assets are the 10 largest constituents of the S\&P~500 by market capitalization at the end of 2022: \texttt{AAPL}, \texttt{MSFT}, \texttt{AMZN}, \texttt{GOOG}, \texttt{GOOGL}, \texttt{UNH}, \texttt{JNJ}, \texttt{XOM}, \texttt{BRK.B}, and \texttt{JPM}.  
    Daily adjusted closing prices from January 1, 2022, to March 1, 2022, are obtained from \cite{yahoo_finance_2024}.  
    The risk-free asset is represented by the 1-Month Constant Maturity Market Yield on U.S. Treasury Securities obtained from~\cite{fred}. Because the observed yield varies over time, it is included as an additional return series in the DRO~model.

    We evaluate the models over various ambiguity radii $\varepsilon \in [ 10^{-4}, \ 10^{0}]$ and set the supporting-hyperplane approximation tolerance to~$\eta = 10^{-3}$. The radius grid is used to expose the robustness--performance trade-off.

\paragraph{Solution Fidelity and Implementation Consistency}
    Figure~\ref{figure: Optimal Weights for the 10 Risky and 1 Risk-Free Asset Portfolio.} compares the portfolio weights produced by
    $\texttt{DRO}_{\texttt{HYP}}$,
    $\texttt{DRO}_{\texttt{EX}}$, and
    $\texttt{RSOME}_{\texttt{HYP}}$ across the Wasserstein-radius grid. 
    The left, middle, and right panels correspond to the three methods, respectively.
    The maximum observed robust-value gap between
$\texttt{DRO}_{\texttt{HYP}}$ and
$\texttt{DRO}_{\texttt{EX}}$ is $1.38\times10^{-4}$, below the prescribed approximation tolerance
$\eta=10^{-3}$. The maximum absolute robust-value discrepancy between $\texttt{DRO}_{\texttt{HYP}}$ and
$\texttt{RSOME}_{\texttt{HYP}}$ is~$1.96\times10^{-12}$. Thus, the exact benchmark confirms the certified
approximation fidelity, while the \texttt{RSOME} comparison confirms
implementation consistency.

\begin{figure}[htbp]
    \centering
    \includegraphics[width=.8\linewidth]{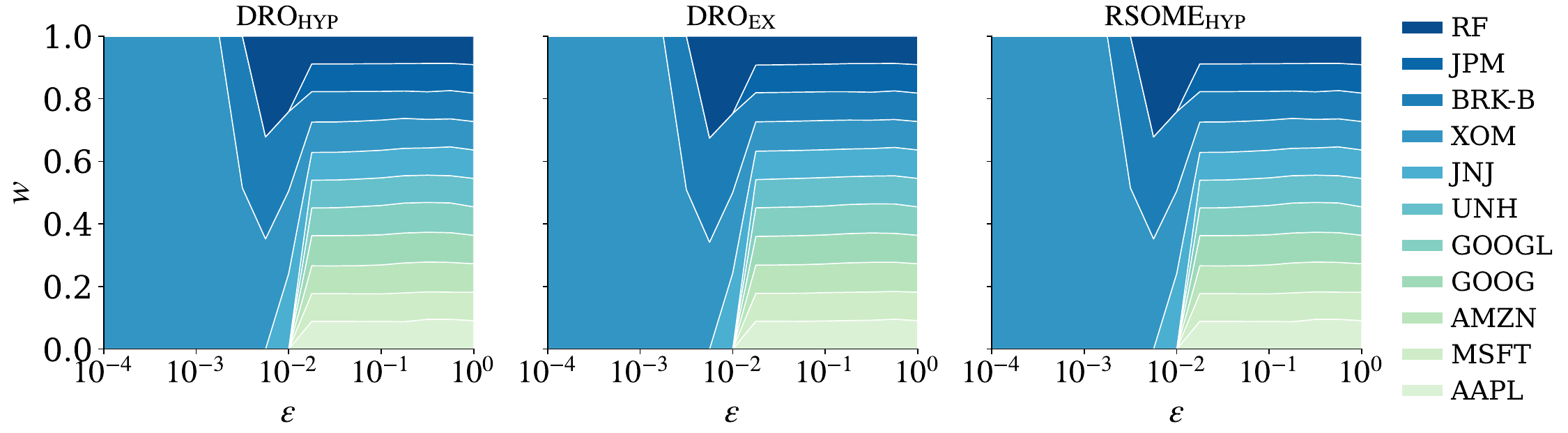}
    \caption{Optimal Portfolio Weights for 10 Risky Assets and 1 Risk-Free Asset.}
    \label{figure: Optimal Weights for the 10 Risky and 1 Risk-Free Asset Portfolio.}
\end{figure}

\paragraph{Computation Time}
    Table~\ref{table: computation time (single period)} reports the mean end-to-end solution time over the Wasserstein-radius grid for the 11-asset single-period benchmark. For $\texttt{DRO}_{\texttt{HYP}}$, we solve the box-specialized LP in Corollary~\ref{corollary: box-specialized HYP reformulation}. The proposed $\texttt{DRO}_{\texttt{HYP}}$ formulation requires $0.009$ seconds on average, compared with~$3.702$ seconds for the exact $\texttt{DRO}_{\texttt{EX}}$ and $0.504$ seconds for $\texttt{RSOME}_{\texttt{HYP}}$. Thus, $\texttt{DRO}_{\texttt{HYP}}$ is approximately~$413$ times faster than the exact sample-specific vertex formulation and $56$ times faster than the \texttt{RSOME} implementation of the same~surrogate.

\begin{table}[htbp]
	\centering
    \scriptsize
	\caption{Solution fidelity and mean end-to-end computation time for the
	11-asset benchmark over
	$\varepsilon\in[10^{-4},10^0]$ with $\eta=10^{-3}$.
	Value gaps are measured relative to
	$\texttt{DRO}_{\texttt{HYP}}$.}
	\label{table: computation time (single period)}
	\begin{tabular}{lrr}
	\toprule
	Method & Mean time (s) & Max. abs. value gap \\
	\midrule
	$\texttt{DRO}_{\texttt{HYP}}$
		& 0.009 & -- \\
	$\texttt{DRO}_{\texttt{EX}}$
		& 3.702 & $1.38\times10^{-4}$ \\
	$\texttt{RSOME}_{\texttt{HYP}}$
		& 0.504 & $1.96\times10^{-12}$ \\
	\bottomrule
	\end{tabular}
\end{table}

\subsection{Multiple-Period Large-Scale Portfolio Rebalancing}
\label{sec: Multiple-Period Case}
    In the multiple-period experiment, we evaluate repeated deployment of $\texttt{DRO}_{\texttt{HYP}}$ for large-scale portfolio rebalancing under varying levels of ambiguity. For the long-only portfolios, box return support, and $\ell_1$ ground norm used in this experiment, we implement the general hyperplane--dual formulation through its exact box-specialized LP representation in Corollary~\ref{corollary: box-specialized HYP reformulation}.

\paragraph{Data and Experimental Setup}
    We consider a five-year out-of-sample trading period from January~2021 through December~2025. This horizon provides~60 consecutive monthly rebalancing decisions and spans heterogeneous return and volatility conditions, allowing us to evaluate repeated controller deployment beyond a single market episode. 

    To isolate computational scalability from missing-data, we construct a fixed universe of~475 risky assets with complete daily adjusted closing-price records over December~2020--December~2025; the price data are obtained from~\cite{yahoo_finance_2024}. 
    Together with the risk-free proxy used in Section~\ref{sec: Single-Period Cases}, the optimization problem has $n=476$ assets. Because the complete-data universe is selected using the full study period, this experiment evaluates repeated high-dimensional controller deployment rather than the performance of a point-in-time S\&P~500 constituent strategy.
    We use a Wasserstein-radius grid with $\varepsilon \in [10^{-4}, 10^{0}]$ and set the supporting-hyperplane approximation tolerance to~$\eta = 10^{-3}$, consistent with the single-period benchmark. As above, the radius grid is used for sensitivity analysis rather than as a universal calibration rule.

    We implement the following rolling procedure:
    $(i)$~For each calendar month from December~2020 through November~2025, $\texttt{DRO}_\texttt{HYP}$ computes one optimal portfolio-weight vector using that month's daily returns, yielding 60 monthly decisions.
    $(ii)$~Each portfolio-weight vector is applied during the immediately following calendar month; for example, the vector computed from December~2020 returns is applied in January~2021. This procedure produces 60 out-of-sample monthly holding periods ending in December~2025.
 
    We evaluate the rolling implementation using computation time, account-value trajectories, and the performance metrics defined below.

\paragraph{Computation Time}
    Across the 300 solves corresponding to 60 monthly instances and five Wasserstein radii, the mean and median solution times were $0.065$ and $0.044$ seconds, respectively; the $90$th percentile was
    $0.137$ seconds, and the maximum was $0.534$ seconds.
    Thus, the $\texttt{DRO}_\texttt{HYP}$ model is readily compatible with monthly rebalancing of the S\&P~500 portfolio.

\paragraph{Performance Metrics}
Let $T_{\rm d}$ denote the number of out-of-sample trading days and $V(t)$ the account value on day~$t$. Let $\bar r$ and $\sigma_{\rm daily}$ denote the mean and standard deviation of daily portfolio returns, respectively, and let $r_f$ denote the daily risk-free rate. We report cumulative return ($\mathrm{CR}$), annualized volatility ($\sigma$), annualized Sharpe ratio ($\mathrm{SR}$), maximum drawdown ($\mathrm{MDD}$), and the Calmar ratio:
{\small \begin{align*}
    &\mathrm{CR}
    :=\frac{V(T_{\rm d})-V(0)}{V(0)},
     &&\qquad
    \sigma
    :=\sqrt{252}\,\sigma_{\rm daily},
    \\
    &\mathrm{SR}
    :=\sqrt{252}\,\frac{\bar r-r_f}{\sigma_{\rm daily}},
    &&\qquad
    \mathrm{MDD}
    :=\max_{0\leq l<k\leq T_{\rm d}}
    \frac{V(l)-V(k)}{V(l)},
    \\
    &r_{\rm annualized}
    :=\left(\frac{V(T_{\rm d})}{V(0)}\right)^{252/T_{\rm d}}-1,
    &&\qquad
    \mathrm{Calmar}
    :=\frac{r_{\rm annualized}}{\mathrm{MDD}}.
\end{align*}
}

\paragraph{Out-of-Sample Trading Performance}
    We evaluate the out-of-sample trading performance of the $\texttt{DRO}_\texttt{HYP}$ model for $\varepsilon \in \{10^{-4}, 10^{-3}, 10^{-2}, 10^{-1}, 10^{0} \}$.
    To provide benchmarks for comparison, we consider three reference portfolios constructed from the same set of S\&P 500 constituents: the sample average approximation (\texttt{SAA}) portfolio, the equal-weight buy-and-hold (\texttt{EW(BH)}) portfolio, and the SPY ETF buy-and-hold (\texttt{SPY(BH)}) portfolio.
    The rolling out-of-sample trading period spans from January~1 2021 through December~31 2025.
    Figure~\ref{figure: account_trajectory_DROHYP_SP500_2021.01~2025.12} reports the corresponding account-value trajectories for different choices of $\varepsilon$ alongside the \texttt{SAA}, \texttt{EW(BH)}, and \texttt{SPY(BH)} benchmarks.
    Table~\ref{table: Mult Performance Metrics of SP500} summarizes the performance metrics.

The Wasserstein radius materially changes the risk--return profile. At
$\varepsilon=10^{-4}$, the
$\texttt{DRO}_{\texttt{HYP}}$ portfolio closely resembles \texttt{SAA}:
their cumulative returns are $4.77$ and $5.18$, respectively, while
both exhibit annualized volatility of approximately $0.46$ and maximum
drawdown of approximately $0.38$. At~$\varepsilon=10^{-2}$, the robust
portfolio attains the highest Sharpe ratio and Calmar ratio among the
tested DRO policies, with $\mathrm{SR}=1.11$ and
$\mathrm{Calmar}=1.34$. At~$\varepsilon=1$, its cumulative return, volatility, and maximum drawdown are~$0.81$, $0.16$, and $0.20$, respectively, close to the corresponding \texttt{EW(BH)} values of~$0.83$, $0.16$, and $0.20$. 
Thus, the experiment shows that moderate ambiguity yields the highest observed risk-adjusted performance, whereas large ambiguity produces a substantially more conservative allocation.

\begin{figure}[h!]
    \centering
    \includegraphics[width=.7\linewidth]{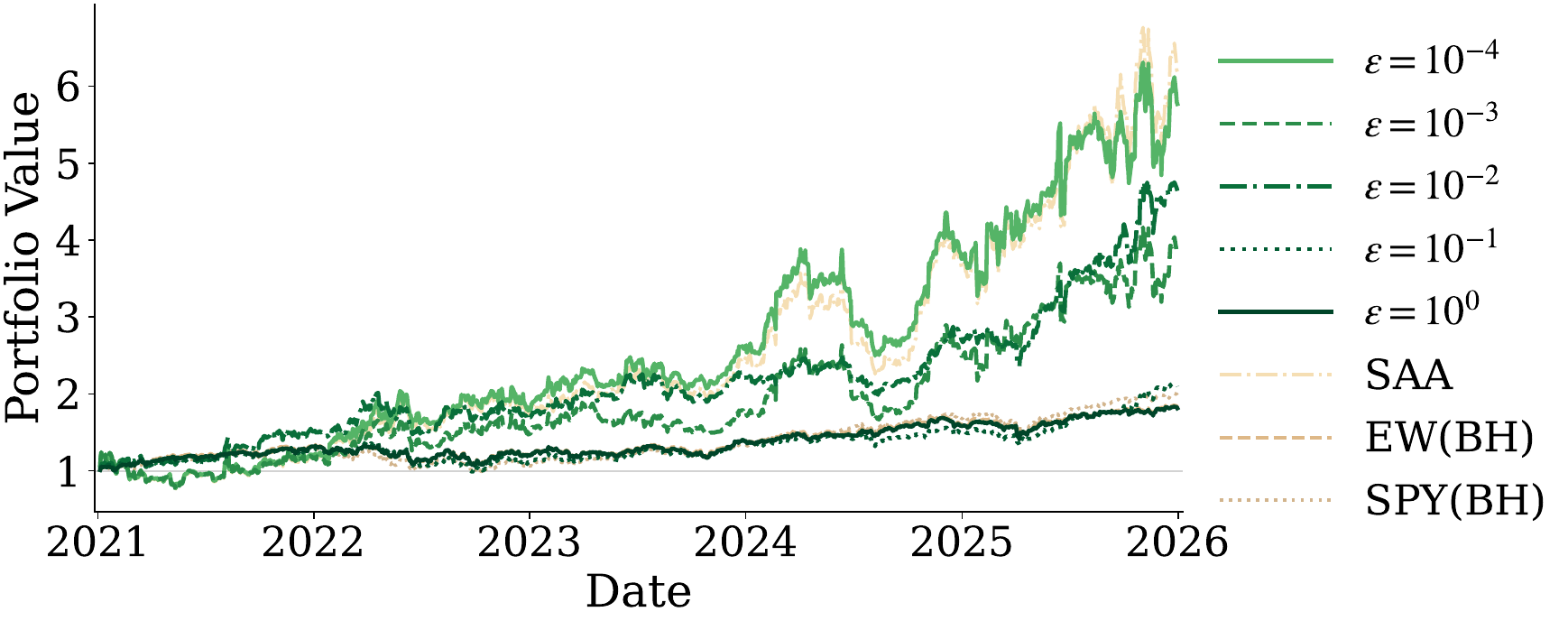}
    \caption{Out-of-Sample Account-Value Trajectories of the S\&P~500 Portfolio with Time-Varying Weights over 2021--2025.}
    \label{figure: account_trajectory_DROHYP_SP500_2021.01~2025.12}
\end{figure}

\begin{table}[h!]
    \centering
    \scriptsize
    \caption{Performance metrics for the fixed 476-asset rolling portfolio
    experiment over 2021--2025.}
    \label{table: Mult Performance Metrics of SP500}
    \begin{tabular}{lrrrrr}
    \toprule
    & CR & $\sigma$ & SR & MDD & Calmar \\
    \midrule
    $\varepsilon=10^{-4}$ & 4.77 & 0.45 & 0.93 & 0.38 & 1.12 \\
    $\varepsilon=10^{-3}$ & 2.81 & 0.43 & 0.76 & 0.38 & 0.80 \\
    $\varepsilon=10^{-2}$ & 3.64 & 0.29 & 1.11 & 0.27 & 1.34 \\
    $\varepsilon=10^{-1}$ & 1.10 & 0.19 & 0.71 & 0.28 & 0.57 \\
    $\varepsilon=10^{0}$  & 0.81 & 0.16 & 0.63 & 0.20 & 0.64 \\
    \texttt{SAA}           & 5.18 & 0.46 & 0.95 & 0.38 & 1.15 \\
    \texttt{EW(BH)}        & 0.83 & 0.16 & 0.63 & 0.20 & 0.66 \\
    \texttt{SPY(BH)}       & 0.98 & 0.17 & 0.70 & 0.25 & 0.60 \\
    \bottomrule
    \end{tabular}
\end{table}

\paragraph{Transaction-Cost Sensitivity}
    As an additional robustness check, we apply proportional transaction-cost rates of $0.1\%$, $0.2\%$, and $0.3\%$ to risky-asset turnover at each rebalancing date. These costs are assessed ex post and are not included in the portfolio-optimization objective. Increasing transaction costs reduces cumulative returns, Sharpe ratios, and Calmar ratios most noticeably for small values of~$\varepsilon$, whose allocations rebalance more aggressively, whereas the portfolio with~$\varepsilon=10^{-2}$ retains the highest Sharpe and Calmar ratios among the tested rebalanced policies at every cost level. Complete account trajectories and performance metrics are reported in Appendix~\ref{appendix: transaction cost sensitivity}.

\subsection{Computational Scalability}
\label{sec: Computational Scalability}

The preceding experiment demonstrates repeated deployment of $\texttt{DRO}_{\texttt{HYP}}$ in an S\&P~500 portfolio.
To isolate computational scaling from variation in market data and monthly sample sizes, we next consider a controlled family of synthetic instances and compare $\texttt{DRO}_{\texttt{HYP}}$ with an
implementation of the exact sample-specific vertex reformulation in Proposition~\ref{prop: exact sample vertex reformulation}. This exact benchmark, denoted by $\texttt{DRO}_{\texttt{EX-CG}}$, uses a constraint-generation cutting-plane algorithm that iteratively adds violated sample--vertex constraints to a restricted master problem. 

This procedure avoids enumerating all $N2^n$ constraints in advance.
Appendix~\ref{appendix: exact constraint generation} details its
implementation and proves that the required samplewise constraint
search is exact.

\paragraph{Data and Experimental Setup}
We fix the sample size at $N=20$, the Wasserstein radius at~$\varepsilon=10^{-2}$, the approximation tolerance at $\eta=10^{-3}$, and the box return support at
$\mathfrak{X}=[-0.15,0.15]^n$, which is described by $2n$ linear inequalities. 
The resulting logarithmic-utility approximation uses
$M:=M_\eta^*=5$ supporting hyperplanes. Hence, with $N=20$, $M=5$, and $r_{\mathfrak{X}}=2n$, the box-specialized $\texttt{DRO}_{\texttt{HYP}}$ formulation in Corollary~\ref{corollary: box-specialized HYP reformulation} contains
$
    n+1+N+nM=6n+21 
$
scalar decision variables. Excluding the constraints defining
$\mathcal W$ and $\lambda\geq0$, it contains $NM=100$
robust-value constraints, $nM=5n$ linking inequalities, and
$nM=5n$ nonnegativity constraints for the variables~$s^m$.
Thus, with $N$ and $M$ fixed, both the variable and constraint
counts grow linearly with~$n$.

Consider eight portfolio dimensions:
$
    n\in\{10,25,50,100,250,476,750,1000\}.
$
We repeat the entire data-generation procedure independently ten
times. In each replication, we first draw asset-specific parameters
$(\mu_i,\sigma_i)$ for $i=1,\dots,1000$, where
$
    \mu_i \sim \operatorname{Unif}[-0.002, 0.002],
$
and
$
    \sigma_i \sim \operatorname{Unif}[0.005, 0.03].
$
Conditional on these parameters, we generate $X_{ji}\sim\mathcal N(\mu_i,\sigma_i^2)$ for $j=1,\dots,N,$ and $i=1,\dots,1000$, and clip the resulting returns to $[-0.15,0.15]$.
Within each replication, the instance of dimension~$n$ consists of the first $n$ columns, so the instances are nested across dimensions.

Both $\texttt{DRO}_{\texttt{HYP}}$ and
$\texttt{DRO}_{\texttt{EX-CG}}$ are implemented in \texttt{CVXPY}
and solved with MOSEK. For $\texttt{DRO}_{\texttt{HYP}}$, we use the
box-specialized linear formulation in
Corollary~\ref{corollary: box-specialized HYP reformulation}.
Each $\texttt{DRO}_{\texttt{EX-CG}}$ restricted master problem retains
the original nonlinear logarithmic utility and is represented as an
exponential-cone program. For both methods, the reported end-to-end
time includes method-specific preprocessing, model construction, and
all solver calls; for $\texttt{DRO}_{\texttt{EX-CG}}$, it also includes
all constraint-generation iterations and samplewise constraint~searches.

\paragraph{Results and Discussion}

Table~\ref{table: computational scalability} shows that
$\texttt{DRO}_{\texttt{HYP}}$ was faster at every tested dimension. As~$n$ increased from 10 to 1000, its median end-to-end time increased
from~$0.0078$ to $0.1279$ seconds, compared with~$0.1253$ to~$4.6461$ seconds for $\texttt{DRO}_{\texttt{EX-CG}}$. 
At the S\&P~500-scale dimension~$n=476$, the corresponding median times were~$0.0580$ and~$1.3076$ seconds, yielding a $22.53\times$ HYP speedup; at $n=1000$, the HYP speedup was $36.32\times$. Across all instances, the largest robust-value gap was~$2.29\times10^{-4}$, well below the prescribed tolerance~$\eta=10^{-3}$, consistent with the value certificate in Proposition~\ref{prop: robust value approximation}. 
Thus, across the tested instances, $\texttt{DRO}_{\texttt{HYP}}$ retains a substantial computational advantage while satisfying its certified approximation~tolerance.

\begin{table}[htbp]
    \centering
    \scriptsize
    \caption{Controlled scaling comparison of
    $\texttt{DRO}_{\texttt{HYP}}$ and
    $\texttt{DRO}_{\texttt{EX-CG}}$ over ten replications ($N=20$, $\eta=10^{-3}$, and
    $\varepsilon=10^{-2}$).
    Times are median end-to-end seconds. HYP speedup denotes
    $\texttt{DRO}_\texttt{EX-CG}/\texttt{DRO}_\texttt{HYP}$, and the value gap is~$\widehat V_\eta^\star-V^\star$.}
    \label{table: computational scalability}
    \setlength{\tabcolsep}{3.5pt}
    \begin{tabular}{l cc c c}
        \toprule
        $n$ & $\texttt{DRO}_\texttt{HYP}$ & $\texttt{DRO}_\texttt{EX-CG}$
        & HYP speedup & Max. gap \\
        \midrule
        10   & 0.0078 & 0.1253 & $16.01\times$ & $1.36\times10^{-4}$ \\
        25   & 0.0096 & 0.1448 & $15.09\times$ & $1.22\times10^{-4}$ \\
        50   & 0.0117 & 0.1311 & $11.18\times$ & $2.29\times10^{-4}$ \\
        100  & 0.0172 & 0.1838 & $10.69\times$ & $2.08\times10^{-4}$ \\
        250  & 0.0321 & 0.6561 & $20.43\times$ & $1.32\times10^{-4}$ \\
        476  & 0.0580 & 1.3076 & $22.53\times$ & $1.13\times10^{-4}$ \\
        750  & 0.0887 & 2.4381 & $27.48\times$ & $7.11\times10^{-5}$ \\
        1000 & 0.1279 & 4.6461 & $36.32\times$ & $6.80\times10^{-5}$ \\
        \bottomrule
    \end{tabular}
\end{table}

\subsection{Impact of Wasserstein Radius}
\label{sec: Impact of Wasserstein Radius}
     This subsection investigates how the Wasserstein radius
$\varepsilon$ affects the out-of-sample performance of
$\texttt{DRO}_{\texttt{HYP}}$. For an in-sample data set
$\widehat{\mathfrak{X}}_N
:=\{\widehat{x}_1,\dots,\widehat{x}_N\}$ of $N$ i.i.d.\ observations, let
$\widehat V_{\eta,N}^\star(\varepsilon)$ and
$\widehat w_{\eta,N}(\varepsilon)$ denote the optimal value and an
optimizer, respectively, of the supporting-hyperplane surrogate with
tolerance $\eta$. Its true out-of-sample performance is
\[
    g\bigl(\widehat w_{\eta,N}(\varepsilon)\bigr)
    :=
    \mathbb E^{\mathbb F}
    \left[
    \log\!\left(
    1+\left\langle\widehat w_{\eta,N}(\varepsilon),X\right\rangle
    \right)
    \right].
\]
We define its \emph{certified reliability} as
$
\mathbb{P}^N
\left\{
\widehat{\mathfrak{X}}_N:
g\bigl(\widehat w_{\eta,N}(\varepsilon)\bigr)
\geq
\widehat V_{\eta,N}^\star(\varepsilon)-\eta
\right\},
$
where $\mathbb{P}^N$ is the $N$-fold product distribution induced by the
true return distribution $\mathbb F$.

\paragraph{Synthetic Data}
    We consider a synthetic portfolio of~$n=10$ assets. For each $i=1,\dots,n$, the daily returns of asset~$i$ are generated independently from a normal distribution with location parameter
    $
        \mu_i=-0.003+0.0008i, 
    $
    and scale parameter
    $
        \sigma_i=0.002+0.003i,
    $
    truncated to~$[-0.15,0.15]$. Thus, higher-index assets have both larger location and scale~parameters.

    We solve the $\texttt{DRO}_\texttt{HYP}$ model using training samples of sizes $N \in \{ 10, 100, 500 \}$ and Wasserstein radii $\varepsilon \in [10^{-4}, 1]$ to examine their effects on out-of-sample performance. 
    The experiment is repeated over $500$ independent simulation runs, and we report the average optimal weights, as well as out-of-sample performance and reliability evaluated using 1,000 independent test samples drawn from the generating distribution in each run.

\paragraph{Average Portfolio Weights}
    Figure~\ref{figure: average_opt_weight} shows the average portfolio weights across $500$ simulation runs. Assets are ordered by index, with higher-index (higher-mean) assets appearing at the top.
    As the sample size $N$ increases, the empirical distribution converges to the true distribution, and at the same time, portfolios computed with smaller $\varepsilon$ tend to allocate heavily to large-index assets. 

\begin{figure}[h!]
    \centering
    \includegraphics[width=.7\linewidth]{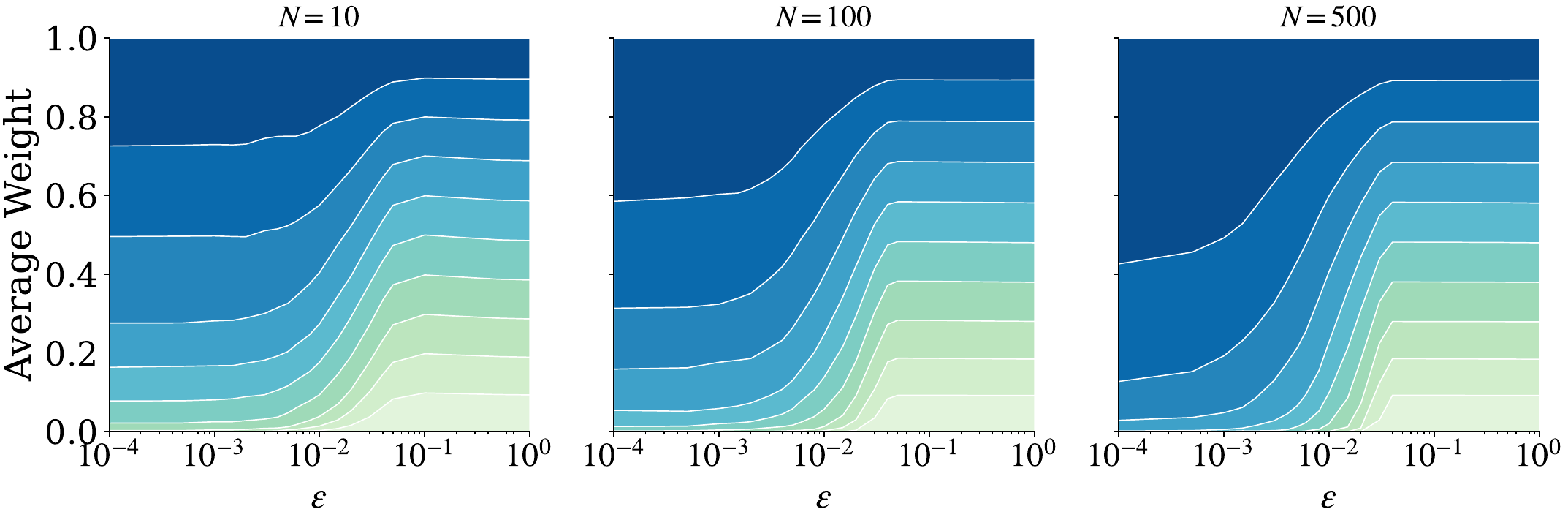}
    \caption{Average Portfolio Weights.}
    \label{figure: average_opt_weight}
\end{figure}

\paragraph{Out-of-Sample Performance}
    Out-of-sample performance was estimated by evaluating each simulation on 1,000 independent testing samples. 
    Figure~\ref{figure: out_of_sample_performance} shows the results averaged across simulation runs, where the lines represent the out-of-sample performance of the $\texttt{DRO}_\texttt{HYP}$ models, whereas the hollow circles indicate the corresponding SAA performance.

\begin{figure}[h!]
    \centering
    \includegraphics[width=0.5\linewidth]{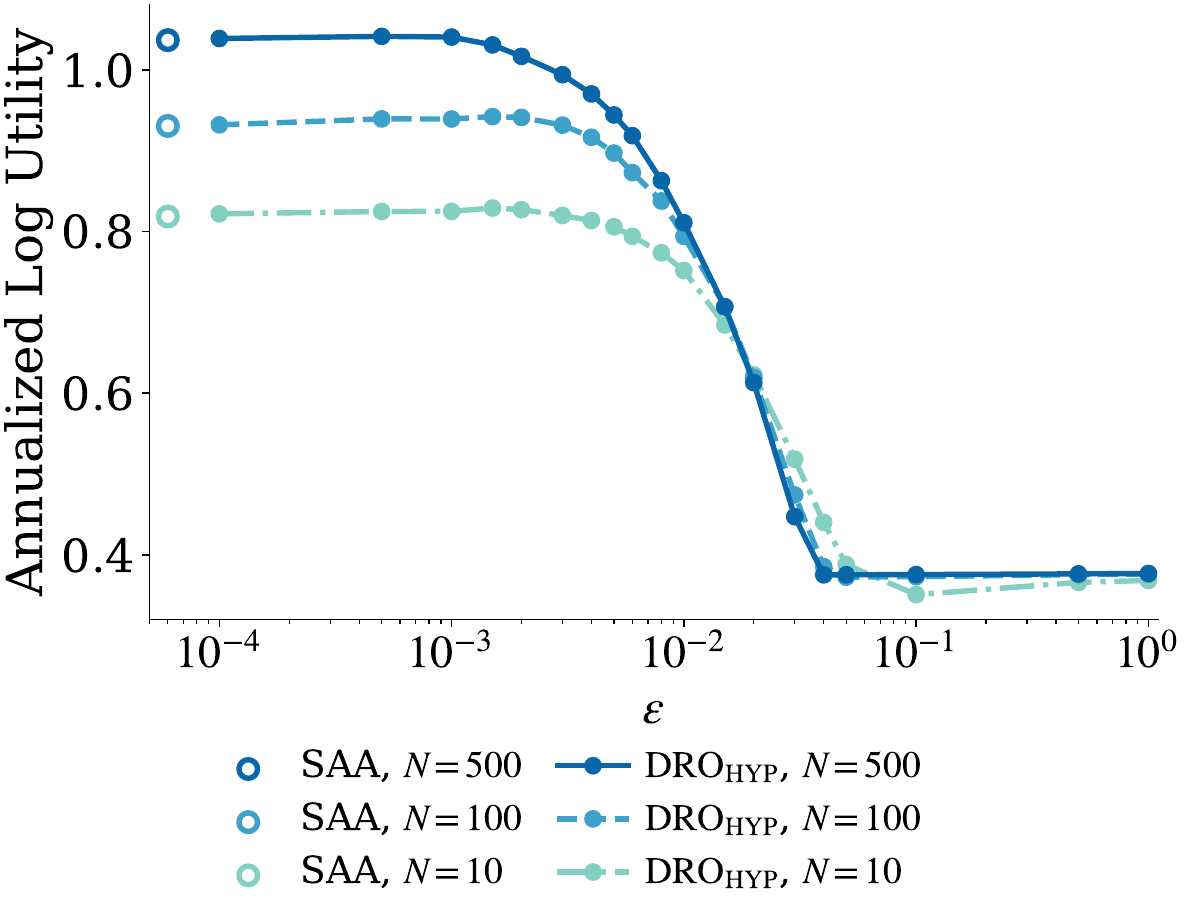}
    \caption{Out-of-Sample Performance versus $\varepsilon$.}
    \label{figure: out_of_sample_performance}
\end{figure}

The $y$-axis shows annualized expected log-growth, obtained by multiplying the estimated daily expected log-growth rate by 252. 
For every sample size, the portfolio computed at the smallest radius performs similarly to its SAA counterpart. As $\varepsilon$ increases
from its smallest value, out-of-sample performance initially improves,
reaches an interior maximum, and then declines as the resulting portfolio becomes increasingly conservative. Increasing the sample size generally improves out-of-sample performance, although this benefit becomes limited at large radii, where the portfolio is governed primarily by worst-case support values.

\paragraph{Reliability}
Figure~\ref{figure: reliability} reports the Monte Carlo estimate of
the certified reliability defined in
Corollary~\ref{corollary: out-of-sample reliability}, with~$g(\widehat w_{\eta,N}(\varepsilon))$ estimated using 1,000 independent
test observations in each simulation run. The estimated reliability
increases with both $\varepsilon$ and~$N$. This pattern is consistent
with Corollary~\ref{corollary: out-of-sample reliability}: increasing
$\varepsilon$ enlarges the event
$d_p(\mathbb F,\widehat{\mathbb F}_N)\leq\varepsilon$, while increasing
$N$ improves the concentration of $\widehat{\mathbb F}_N$ around~$\mathbb F$.

\begin{figure}[htbp]
    \centering
    \includegraphics[width=0.5\linewidth]{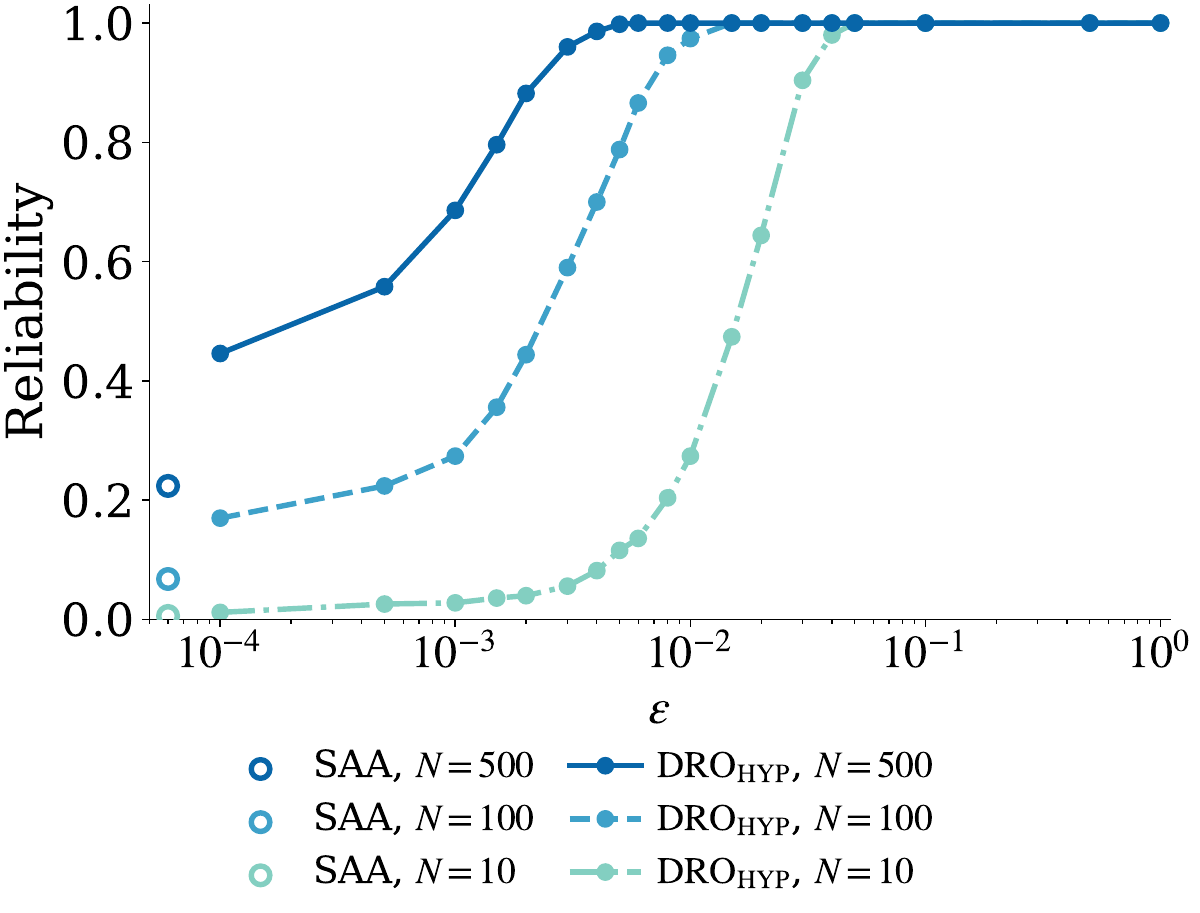}
    \caption{Reliability versus $\varepsilon$.}
    \label{figure: reliability}
\end{figure}

\paragraph{Practical Considerations}
    Increasing $\varepsilon$ introduces a trade-off between expected growth and reliability. In these experiments, small-to-moderate radii yield higher out-of-sample performance, whereas larger radii provide greater reliability. Accordingly, investors can select $\varepsilon$ to balance growth and reliability according to their preferences.

\section{Conclusion}
\label{Section: Conclusion}
    This paper developed a certified framework for high-dimensional Wasserstein robust expected-utility portfolio optimization. Starting from the semi-infinite dual reformulation, we combined a supporting-hyperplane utility surrogate with duality over compact polyhedral return supports to obtain a finite hyperplane--dual convex program. Under an $\ell_1$ ground norm and polyhedral portfolio constraints, the resulting formulation is a single linear program whose numbers of variables and constraints grow polynomially with the problem dimensions.

    The formulation is certified: the uniform utility-approximation tolerance bounds both the robust-value error and the near-optimality gap for the original Wasserstein DRO objective.
    For box support, we also established a finite-threshold large-ambiguity characterization with a closed-form optimizer set. Computational studies showed that $\texttt{DRO}_{\texttt{HYP}}$ was faster than the exact constraint-generation benchmark at every tested dimension up to $n=1000$, while the observed robust-value gaps remained below the prescribed approximation tolerance. They further demonstrated monthly S\&P~500-scale rebalancing and quantified the robustness--performance trade-off governed by the Wasserstein radius.



\medskip
\noindent\textbf{Acknowledgment.}
{
		This paper was supported in part by the National Science and Technology Council (NSTC), Taiwan, under Grants: NSTC113--2628--E--007--015-- and NSTC114--2628--E-007--006--.
}

\normalsize
\clearpage
\appendix
\renewcommand{\thetheorem}{\Alph{section}.\arabic{theorem}}

\medskip
\section{Technical Proofs}
\label{appendix: technical proofs}

\begin{proof}[Proof of Proposition~\ref{prop: exact sample vertex reformulation}] 
    By Lemma~\ref{lemma: A Convex Reformulation of the DRO Problem}, the DRO problem~\eqref{eq: DRO problem} is equivalent to~\eqref{eq: convex DRO}. 
    Fix $w\in\mathcal{W}$, $\lambda\geq0$, and $j \in \{1,\dots, N\}$, and define a shorthand
    \[
    \phi_j(x):=
    U(\langle w,x\rangle)+\lambda\|x-\widehat{x}_j\|_1.
    \]
    For any $x\in\mathfrak{X}$, let
    $\widetilde x:=\min\{x, \widehat{x}_j\}$ componentwise. Since
    $w\geq0$ and $U(\cdot)$ is increasing,
    $U(\langle w,\widetilde x\rangle)\leq U(\langle w,x\rangle)$.
    Moreover,
    $\|\widetilde x-\widehat{x}_j\|_1
    \leq\|x-\widehat{x}_j\|_1$.
    Thus, the minimum of $\phi_j$ over $\mathfrak{X}$ can be restricted to
    $$
    [x_{\min},\widehat{x}_j]:=\{x \in \mathbb{R}^n: x_{\min} \leq x \leq \widehat{x}_j \}.
    $$
     That is,
    \[
    \min_{x \in \mathfrak{X}} \phi_j(x) = \min_{x \in [x_{\min}, \widehat{x}_j]} \phi_j(x).
    \]
    On this restricted box $[x_{\min},\widehat{x}_j]$, it follows that
    $
    \|x-\widehat{x}_j\|_1
        = \langle \mathbf 1,\widehat{x}_j-x\rangle,
    $
    which is affine in $x$. Hence $\phi_j$ is concave on
    $[x_{\min},\widehat{x}_j]$, and its minimum is attained at an extreme
    point of this box. These extreme points are precisely the elements of
    $\mathcal V_j$.

    Indeed, if \(x=\sum_{\ell}\theta_\ell v^\ell\) is a convex
    combination of vertices of \([x_{\min},\widehat{x}_j]\), then concavity gives
    \[
    \phi_j(x) = \phi_j \left( \sum_{\ell}\theta_\ell v^\ell \right)
    \geq
    \sum_\ell \theta_\ell \phi_j(v^\ell)
    \geq
    \min_\ell \phi_j(v^\ell).
    \]
    Hence,
    \[
    \min_{x\in[x_{\min},\widehat{x}_j]}\phi_j(x) = \min_{v\in\mathcal V_j}\phi_j(v).
    \]
    Substitution into~\eqref{eq: convex DRO} proves the result.
\end{proof}

\begin{proof}[Proof of Theorem~\ref{theorem: DRO via hyperplane}]
    Recall that
    $
    \widehat{U}_{M_\eta^*}(y)
        := \min_{1\leq m \leq M_\eta^*}(\alpha_m y+\beta_m)
    $
    is a supporting-hyperplane surrogate of $U$ with uniform error at most~$\eta$ on $[\underline{y},\overline{y}]$. By applying the semi-infinite reformulation in Lemma~\ref{lemma: A Convex Reformulation of the DRO Problem} to the surrogate Wasserstein DRO problem~\ref{problem: supporting hyperplane DRO surrogate}, we obtain
    \begin{align}
        &\begin{cases}
            \displaystyle\sup_{w \in \mathcal{W}, \lambda \geq 0, a_j, \forall j} - \lambda \varepsilon +  \frac{1}{N} \sum_{j=1}^{N} a_j\\
            {\rm s.t.} \;  
                \displaystyle \inf_{x \in \mathfrak{X}} \left[
                    \min_{1\leq m \leq M_\eta^*} \left( \alpha_m \langle w ,x \rangle + \beta_m \right) 
                + 
                \lambda \| x-\widehat{x}_j \| \right] \geq a_j, \\
                \qquad j \leq N.
        \end{cases} \nonumber
        \\
        &=\begin{cases}
            \displaystyle \sup_{w \in \mathcal{W}, \lambda \geq 0, a_j, \forall j} - \lambda \varepsilon +  \frac{1}{N} \sum_{j=1}^{N} a_j\\
            {\rm s.t.} \;  
                \displaystyle \min_{1 \leq m \leq M_\eta^*} \inf_{x \in \mathfrak{X}} \left[ \alpha_m \langle w ,x \rangle + \beta_m  
                + 
                \lambda \| x-\widehat{x}_j \| \right] \geq a_j, \\
                \qquad j \leq N.
        \end{cases} \nonumber
        \\
        &=\begin{cases}
            \displaystyle\sup_{w \in \mathcal{W}, \lambda \geq 0, a_j} - \lambda \varepsilon +  \frac{1}{N} \sum_{j=1}^{N} a_j\\
            {\rm s.t.} \;
                \displaystyle\inf_{x \in \mathfrak{X}} \alpha_m \langle w ,x \rangle + \beta_m  + \lambda  \| x - \widehat{x}_j \|  \geq a_j, \\
                \qquad j \leq N, \ m \leq M_\eta^*.
        \end{cases} \label{eq: hyperplane over m} 
    \end{align}
    The last equivalence follows since the lower bound on the finite minimum over $m$ is equivalent to the same lower bound holding for each $m=1,\dots,M_\eta^*$.

    Using the fact that the double dual of a norm equals the original norm in a finite-dimensional space, i.e., $\|a\|_{**} = \|a\|$ for $a \in \mathbb{R}^n$, together with the dual-norm identity $\max_{z: \|z\|_* \leq \lambda} \langle z, v \rangle = \lambda \|v\|$, we rewrite~\eqref{eq: hyperplane over m} as 
    \begin{align}
        &
    \begin{cases}
        \displaystyle\sup_{w \in \mathcal{W}, \lambda \geq 0, a_j, \forall j} 
            - \lambda \varepsilon +  \frac{1}{N} \sum_{j=1}^{N} a_j\\
        {\rm s.t.} \;  
            \displaystyle\inf_{x \in \mathfrak{X}}  \alpha_m \langle w, x \rangle + \beta_m  
            + 
            \max_{z_j^m: \|z_j^m\|_* \leq \lambda} \langle z_j^m, x-\widehat{x}_j \rangle \geq a_j, \\
            \qquad  j \leq N, \ m \leq M_\eta^*.
    \end{cases}\nonumber
    \\
    & =
    \begin{cases}
        \displaystyle\sup_{w \in \mathcal{W}, \lambda \geq 0, a_j, \forall j} - \lambda \varepsilon +  \frac{1}{N} \sum_{j=1}^{N} a_j\\
        {\rm s.t.} \;  
            \displaystyle\max_{z_j^m: \| z_j^m \|_* \leq \lambda}\inf_{ x \in \mathfrak{X}}  \alpha_m \langle w ,x \rangle + \beta_m  + \langle z_j^m, x-\widehat{x}_j \rangle \geq a_j, \\
            \qquad  j \leq N, \ m \leq M_\eta^*.
    \end{cases}\label{eq: minmax ineq} 
    \end{align}
    The interchange of minimum and maximum in~\eqref{eq: minmax ineq} follows from Sion's minimax theorem; see \cite[Proposition~5.5.4]{bertsekas2009convex}.\footnote{
        Indeed, for fixed $w$, $\lambda$,~$j$, and $m$, the function 
            $
            \alpha_m \langle w ,x \rangle + \beta_m  + \langle z_j^m, x-\widehat{x}_j \rangle
            $
        is affine (hence convex) in~$x$ and affine (hence concave) in~$z_j^m$.
        Moreover, the set $\mathfrak{X}$ is compact and convex, and for each fixed finite $\lambda \geq 0$, the set $\{ z_j^m \in \mathbb{R}^n : \|z_j^m\|_* \leq \lambda \}$ is also a compact convex set.
        Thus, Sion's minimax theorem applies.
    } 
    For each sample--hyperplane pair $(j,m)$, the maximum over~$z_j^m$ is attained on the compact dual-norm ball. Therefore, the constraint
    \[
    \max_{\|z_j^m\|_* \leq \lambda} \, \inf_{x \in \mathfrak{X}} F_{jm}(x, z_j^m)\geq a_j
    \]
    where 
    $
    F_{jm}(x, z_j^m):=\alpha_m \langle w ,x \rangle + \beta_m  + \langle z_j^m, x-\widehat{x}_j \rangle
    $
    is equivalent to the existence of a vector $z_j^m$ such~that
    \[
    \|z_j^m\|_*\leq\lambda
    \; \text{ and } \; 
    \inf_{x\in\mathfrak{X}}F_{jm}(x,z_j^m)\geq a_j.
    \]
    Introducing these auxiliary variables yields:
    \begin{align}
        & 
        \begin{cases}
            \displaystyle\sup_{w \in \mathcal{W}, \lambda \geq 0, a_j, z_j^m, \forall j, m} - \lambda \varepsilon +  \frac{1}{N} \sum_{j=1}^{N} a_j\\
            {\rm s.t.} \;  
                \displaystyle\inf_{x \in \mathfrak{X}}   \alpha_m \langle w ,x \rangle + \beta_m  +  \langle z_j^m, x-\widehat{x}_j \rangle  \geq a_j, \\
                \qquad \quad j \leq  N, \ m \leq  M_\eta^*;\\
                \quad\;\; \| z_j^m \|_* \leq \lambda, \quad j \leq  N, \ m \leq M_\eta^*.
        \end{cases} \nonumber \\
        &=
         \begin{cases}
            \displaystyle\sup_{w \in \mathcal{W}, \lambda \geq 0, a_j, z_j^m, \forall j, m} - \lambda \varepsilon +  \frac{1}{N} \sum_{j=1}^{N} a_j\\
            {\rm s.t.} \;  
                \displaystyle\min_{x \in \mathfrak{X}}   \alpha_m \langle w ,x \rangle + \beta_m  +  \langle z_j^m, x-\widehat{x}_j \rangle  \geq a_j, \\
                \qquad \quad j \leq  N, \ m \leq  M_\eta^*;\\
                \quad\;\; \| z_j^m \|_* \leq \lambda, \quad j \leq N, \ m \leq M_\eta^*.
        \end{cases} \label{eq: linear program}
    \end{align}
    Equality~\eqref{eq: linear program} follows from the fact that the map $x \mapsto \alpha_m \langle w, x \rangle + \beta_m + \langle z_j^m, x - \widehat{x}_j \rangle$ is continuous in $x$ over the compact set $\mathfrak{X}$; hence, by the Weierstrass extreme value theorem, the infimum is attained. 
     
    To simplify the minimization over $x$ in the constraints of~\eqref{eq: linear program}, we use the polyhedral representation
    $
    \mathfrak{X} 
    = \{x \in \mathbb{R}^n : H x \leq h \}.
    $
    For each~$j=1,\ldots,N$ and $m=1,\ldots,M_\eta^*$, the inner affine minimization over $\mathfrak{X}$ is
    \begin{align}
        \min_{x\in\mathfrak{X}}
        \big(
            \alpha_m\langle w,x\rangle
            +
            \langle z_j^m,x\rangle
        \big)
        =
        \begin{cases}
            \displaystyle
            \min_x \ \langle \alpha_m w+z_j^m,x\rangle\\
            {\rm s.t.}\quad Hx\le h .
        \end{cases}
        \label{eq: constraint in polytope}
    \end{align}
For each sample--hyperplane pair $(j, m)$, introduce the multiplier $\mu_j^m \in \mathbb{R}^{r_{\mathfrak{X}}}_+$ associated with the constraints $Hx \leq h$. The corresponding Lagrangian is
\begin{align}
    \mathcal L(x,\mu_j^m)
    &= \langle \alpha_m w+z_j^m,x\rangle
        + \langle \mu_j^m,Hx-h\rangle \notag\\
    &= \langle \alpha_m w+z_j^m+H^\top\mu_j^m,x\rangle
        - \langle \mu_j^m,h\rangle
\end{align}
where the last equality uses the adjoint identity $\langle \mu_j^m, Hx \rangle = \langle H^\top \mu_j^m,x\rangle$. Therefore,
\[
    \inf_x \mathcal L(x,\mu_j^m)
    =
    \begin{cases}
    -\langle \mu_j^m,h\rangle, & \alpha_m w+z_j^m+H^\top\mu_j^m=0,\\
    -\infty, & \text{otherwise}.
    \end{cases}
\]
By strong LP duality, for each fixed sample--hyperplane pair~$(j,m)$, 
\begin{align*}
& \min_{x\in\mathfrak{X}}
\left[
    \alpha_m\langle w,x\rangle+\beta_m
    +
    \langle z_j^m,x-\widehat{x}_j\rangle
\right] 
 \\
 &=
\beta_m-\langle z_j^m,\widehat{x}_j\rangle
+
\max_{\mu_j^m\geq0}\inf_x \mathcal L(x,\mu_j^m)\\
& =
\beta_m-\langle z_j^m,\widehat{x}_j\rangle + \max_{\mu_j^m\geq0}
\{ 
    -\langle \mu_j^m,h\rangle:
    \alpha_m w+z_j^m+H^\top\mu_j^m=0
\}\\
& =
\max_{\mu_j^m\geq0}
\{
    \beta_m-\langle z_j^m,\widehat{x}_j\rangle
    -\langle \mu_j^m,h\rangle:
    \alpha_m w+z_j^m+H^\top\mu_j^m=0
\}.
\end{align*}
Hence the constraint in~\eqref{eq: linear program} is equivalent to the existence of
$\mu_j^m\geq0$ such~that
$
    \alpha_m w+z_j^m+H^\top\mu_j^m=0
$
and
$
    \beta_m-\langle z_j^m,\widehat{x}_j\rangle
    -\langle \mu_j^m,h\rangle
    \geq a_j .
$
Eliminating $z_j^m$ through the equality constraint yields
\[
    \alpha_m\langle w,\widehat{x}_j\rangle+\beta_m
    +
    \langle \mu_j^m,H\widehat{x}_j-h\rangle
    \geq a_j .
\]
Additionally, the norm constraint $\|z_j^m\|_*\leq\lambda$ becomes
$$
    \|-\alpha_m w-H^\top\mu_j^m\|_*\leq \lambda.
$$
Thus the problem becomes
\begin{align}\label{eq: linear program (dual)}
    &\displaystyle
    \sup_{\substack{
    w\in\mathcal{W},\ \lambda\ge0,\\
    a_j,\;
    \mu_j^m\geq 0 \; \forall j,m
    }}
    -\lambda\varepsilon+\frac1N\sum_{j=1}^N a_j \\
    {\rm s.t.}\quad
        &\alpha_m\langle w,\widehat{x}_j\rangle+\beta_m
        +\langle \mu_j^m,H\widehat{x}_j-h\rangle
        \geq a_j,
         \qquad j\leq N,\ m\leq M_\eta^*, \notag\\
        &\|-\alpha_m w-H^\top\mu_j^m\|_*
        \leq \lambda,\qquad j\le N,\ m\le M_\eta^*. \notag
\end{align}
    To complete the proof, it remains to identify the computational class of~\eqref{eq: linear program (dual)}.
The objective is affine, and the first constraint is affine in the decision variables.
The second constraint is
\[
    \|-\alpha_m w-H^\top\mu_j^m\|_q \leq \lambda,
\]
which is a convex norm constraint.
Hence~\eqref{eq: linear program (dual)} is a finite convex program.

    If $p=1$, then $q=\infty$, and the norm constraint is equivalent to the linear inequalities
    \[
        -\lambda
        \leq
        -\alpha_m w_i - (H^\top\mu_j^m)_i
        \leq
        \lambda,
        \quad i=1,\dots,n,
    \]
    where $(H^\top\mu_j^m)_i = \sum_{\ell=1}^{r_{\mathfrak{X}}} H_{\ell i}\, (\mu_j^m)_\ell$, denotes the $i$th component of $H^\top \mu_j^m \in \mathbb{R}^n$.
    Therefore, when $\mathcal{W}$ is polyhedral, the formulation is a robust linear program for the $\ell_1$ ground~norm.
\end{proof}

\begin{proof}[Proof of Corollary~\ref{corollary: box-specialized HYP reformulation}]
The box support is a special case of~\eqref{eq: compact polyhedron} with
\[
    H :=
    \begin{bmatrix}
        I\\-I
    \end{bmatrix},
    \qquad
    h :=
    \begin{bmatrix}
        x_{\max}\\-x_{\min}
    \end{bmatrix}.
\]
For each sample--hyperplane pair $(j,m)$, write the support-dual variable in Theorem~\ref{theorem: DRO via hyperplane} as
\[
    \mu_j^m=
    \begin{bmatrix}
        \gamma_j^m\\ \nu_j^m
    \end{bmatrix},
    \qquad
    \gamma_j^m,\nu_j^m\in\mathbb R_+^n,
\]
where $\gamma_j^m$ and $\nu_j^m$ correspond to the upper and lower
support bounds, respectively. The constraints in
Theorem~\ref{theorem: DRO via hyperplane} become
\begin{align}
    &\alpha_m\langle w,\widehat x_j\rangle+\beta_m
    -\langle\gamma_j^m,x_{\max}-\widehat x_j\rangle
    -\langle\nu_j^m,\widehat x_j-x_{\min}\rangle
    \geq a_j,
    \label{eq: box dual value}\\
    &\|-\alpha_m w-\gamma_j^m+\nu_j^m\|_\infty
    \leq\lambda.
    \label{eq: box dual norm}
\end{align}
The $\ell_\infty$-norm constraint~\eqref{eq: box dual norm} is equivalent to the component-wise inequality
\begin{align} \label{ineq: equivalent ell_inf norm constraint}
    -\lambda\mathbf 1
    \leq
    -\alpha_m w-\gamma_j^m+\nu_j^m
    \leq
    \lambda\mathbf 1.
\end{align}
The left inequality, together with $\gamma_j^m\geq\mathbf 0$, implies
\begin{align} \label{ineq: equivalent ell_inf norm constraint reduction}
    \nu_j^m
    \geq \alpha_m w+\gamma_j^m-\lambda\mathbf 1
    \geq
    \alpha_m w-\lambda\mathbf 1.  
\end{align}
Additionally, since
$
x_{\max}-\widehat x_j\geq\mathbf 0
$
and
$
\widehat x_j-x_{\min}\geq\mathbf 0,
$
the lower bound on $\nu_j^m$ in~\eqref{ineq: equivalent ell_inf norm constraint reduction} implies that the left-hand side of~\eqref{eq: box dual value} is maximized by
\begin{align*} 
    \gamma_j^m=\mathbf{0},
    \qquad
    \nu_j^m=  \max\{ \alpha_m w-\lambda\mathbf 1, 0\}.
\end{align*}
It is readily verified that this choice of multipliers also satisfies~\eqref{eq: box dual norm}. Moreover, it is
independent of the sample index~$j$.
Finally, introduce auxiliary variables $s^m\in\mathbb R_+^n$ satisfying
$
    s^m\geq\alpha_m w-\lambda\mathbf1.
$
Since $\widehat x_j-x_{\min}\geq\mathbf0$ and $s^m$
enters the robust-value constraint through
$
-\langle\widehat x_j-x_{\min},s^m\rangle,
$
any value larger than
$\max\{ \alpha_m w-\lambda\mathbf 1, 0\}$ can only tighten the constraint.
Hence, the smallest feasible choice is attained without loss of
generality, and substitution into~\eqref{eq: box dual value} yields
\eqref{eq: box-specialized HYP reformulation}.
\end{proof}

\begin{proof}[Proof of Proposition~\ref{prop: robust value approximation}]
    For any $w \in \mathcal{W}$, $x \in \mathfrak{X}$, we have $\langle w, x \rangle \in [\underline{y},\overline{y}]$. Hence, by the uniform error bound~\eqref{ineq: eta-uniform bound}, it follows that
    \[
    U(\langle w, x\rangle)
        \leq \widehat{U}_{M_\eta^*}(\langle w,x\rangle)
        \leq U(\langle w,x\rangle) + \eta .
    \]
    Taking expectation with respect to any $\mathbb{F} \in \mathcal{B}_\varepsilon(\widehat{\mathbb{F}})$ gives
    \[
    \mathbb{E}^{\mathbb{F}} \left[U(\langle w, X\rangle)\right]
        \leq \mathbb{E}^{\mathbb{F}} \left[\widehat{U}_{M_\eta^*}(\langle w,X\rangle)\right]
        \leq \mathbb{E}^{\mathbb{F}} \left[U(\langle w,X\rangle)\right]+\eta .
    \]
    Taking the infimum over
    $\mathbb{F}\in\mathcal{B}_\varepsilon(\widehat{\mathbb{F}})$ preserves these inequalities, so for every $w\in\mathcal{W}$,
    \[
    \phi(w)
    \leq \widehat\phi_\eta(w)
    \leq \phi(w)+\eta,
    \]
    where
    \[
    \phi(w):=
    \inf_{\mathbb{F}\in\mathcal{B}_\varepsilon(\widehat{\mathbb{F}})}
    \mathbb E^{\mathbb{F}}
    \left[
    U(\langle w,X\rangle)
    \right]
    \]
    and $\widehat\phi_\eta(w)$ is defined analogously with $U$ replaced by $\widehat{U}_{M_\eta^*}$.
    Taking the supremum over $w \in \mathcal{W}$ yields
    \[
    V^\star(\varepsilon)
    \leq
    \widehat V_{\eta}^\star(\varepsilon)
    \leq
    V^\star(\varepsilon)+ \eta.
    \]
    Finally, if $\widehat w_\eta$ optimizes the surrogate problem, then $\widehat V_{\eta}^\star(\varepsilon) =\widehat\phi_\eta(\widehat{w}_\eta) $ and we have
    \[
    \phi(\widehat{w}_\eta)
    \geq \underbrace{ \widehat\phi_\eta(\widehat{w}_\eta)-\eta}_{  
        = \widehat V_{\eta}^\star(\varepsilon) -\eta}
    \geq V^\star(\varepsilon)-\eta,
    \]
    which proves the near-optimality claim.
\end{proof}

\begin{proof}[Proof of Corollary~\ref{corollary: out-of-sample reliability}]
    On the event
    $d_p(\mathbb F,\widehat{\mathbb F}_N)\leq\varepsilon$, we have
    $\mathbb F\in\mathcal B_\varepsilon(\widehat{\mathbb F}_N)$. Therefore,
    \[
    \begin{aligned}
    g\bigl(\widehat w_{\eta,N}(\varepsilon)\bigr)
    &= \mathbb E^{\mathbb F}[U(\langle \widehat w_{\eta,N}(\varepsilon), X\rangle)]\\
    &\geq
    \inf_{\mathbb Q\in\mathcal B_\varepsilon(\widehat{\mathbb F}_N)}
    \mathbb E^{\mathbb Q}
    \left[
    U\bigl(\langle\widehat w_{\eta,N}(\varepsilon),X\rangle\bigr)
    \right]\\
    &\geq
    \widehat V_{\eta,N}^{\star}(\varepsilon)-\eta,
    \end{aligned}
    \]
    where the second inequality follows from
    Proposition~\ref{prop: robust value approximation}. Taking probabilities
    establishes the result.
\end{proof}

\begin{proof}[Proof of Theorem~\ref{theorem: large ambiguity optimality}] 
For any $\mathbb{F}\in\mathcal M(\mathfrak{X})$, both $\mathbb{F}$ and $\widehat{\mathbb{F}}$ are supported on $\mathfrak{X}$. Hence, by coupling them
arbitrarily,
\[
d_p(\mathbb{F},\widehat{\mathbb{F}})
    \leq \sup_{x, x' \in \mathfrak{X}}\|x - x'\| 
    = \bar{\varepsilon}.
\]
Therefore, for every $\varepsilon \geq \bar{\varepsilon}$,
$
\mathcal{B}_\varepsilon(\widehat{\mathbb{F}})=\mathcal{M}(\mathfrak{X}).
$
Consequently, for each fixed $w\in\Delta_n$ and $\varepsilon\geq\bar\varepsilon$,
\begin{align*}
\inf_{\mathbb{F}\in\mathcal{B}_\varepsilon(\widehat{\mathbb{F}})}
\mathbb E^{\mathbb{F}}\!\left[U(\langle w,X\rangle)\right]
&=
\min_{\mathbb{F}\in\mathcal M(\mathfrak{X})}
\mathbb E^{\mathbb{F}}\!\left[U(\langle w,X\rangle)\right]\\
&=
\min_{x\in\mathfrak{X}}U(\langle w,x\rangle),
\end{align*}
where the last equality follows because a worst-case distribution may be chosen
as a Dirac measure concentrated at a minimizer.

Since $U(\cdot)$ is strictly increasing, we have
  \begin{align*}
        \min_{x \in \mathfrak{X}} U( \langle w, x \rangle )   
         = U \left( \min_{x \in \mathfrak{X}} \langle w, x \rangle   \right),
    \end{align*}
and therefore
\[
\arg\max_{w\in\Delta_n}\min_{x\in\mathfrak{X}} U(\langle w,x\rangle)
=
\arg\max_{w\in\Delta_n}\min_{x\in\mathfrak{X}} \langle w, x \rangle.
\]
Since $w\geq 0$ and $x_i \geq x_{\min, i}$, we have $\sum_{i=1}^n w_i x_i \geq \sum_{i=1}^n w_i x_{\min, i} = \langle w, x_{\min} \rangle $ and the worst case over the box is attained at $x=x_{\min}$:
\[
\min_{x \in \mathfrak{X}} \langle w, x \rangle = \langle w, x_{\min} \rangle
\]
Thus, for every $\varepsilon\geq\bar\varepsilon$, the robust problem has the same optimal solution set as
$
\max_{w\in\Delta_n}\langle w,x_{\min}\rangle .
$
Note that
\begin{align*}
\langle w, x_{\min} \rangle
=
\sum_{i=1}^n w_i x_{\min,i}
\leq
\max_{1\leq k\leq n} x_{\min,k} \sum_{i=1}^n w_i
=
c^*.
\end{align*}
Moreover,
\[
\langle w, x_{\min} \rangle = c^* - c^* + \sum_{i=1}^n w_i x_{\min,i} = c^* - \sum_{i=1}^n w_i (c^* - x_{\min,i}).
\]
Since $w_i\ge0$ and $c^*-x_{\min,i}\ge0$ for all $i$, equality holds if and only if
\[
w_i(c^*-x_{\min,i}) = 0 \quad i=1,\dots,n
\]
Equivalently,
$
w_i = 0$ for all $i \notin \mathcal{I}^*$.
Therefore, we have
\begin{align*}
    \arg\max_{w \in \mathcal{W}} \langle w, x_{\min} \rangle 
    & = \{w \in \mathcal{W}: w_i = 0 \; \text{for all $i \notin \mathcal{I}^*$}\} \\
    &= \operatorname{conv}\{e_i \in \mathbb{R}^n : i \in \mathcal{I}^*\}.
\end{align*}
If $\mathcal I^*=\{i^*\}$, then
\[
W^*=\operatorname{conv}\{e_{i^*}\}=\{e_{i^*}\}.
\]
If $x_{\min}=c\mathbf 1$, then $c^*=c$ and
$\mathcal I^*=\{1,\dots,n\}$, so
\[
W^*
=
\operatorname{conv}\{e_1,\dots,e_n\}
=
\Delta_n.
\]
Therefore, every feasible portfolio is optimal. In particular, the equal-weight portfolio is optimal, but not uniquely selected.
\end{proof}

\begin{proof}[Proof of Corollary~\ref{corollary: Large-Ambiguity Optimality of the Hyperplane Approximation}]
    Define the hyperplane surrogate
    \[
    \widehat{U}_{M_\eta^*}(y):=\min_{1\leq m\leq M_\eta^*}(\alpha_m y+\beta_m).
    \]
    By construction, the finite formulation in Theorem~\ref{theorem: DRO via hyperplane}
    is obtained by replacing $U$ with $\widehat{U}_{M_\eta^*}$ in the semi-infinite convex reformulation.
    Since $\alpha_m>0$ for all $m=1,\dots,M_\eta^*$, let $\alpha_{\min}:=\min_{1\leq m\leq M_\eta^*}\alpha_m>0$.
    For any $y_2>y_1$,
    \[
    \begin{aligned}
    \widehat{U}_{M_\eta^*}(y_2)
    &=
    \min_{1 \leq m \leq M_\eta^*} \{\alpha_m y_1+\beta_m+\alpha_m(y_2-y_1)\}\\
    &\geq \min_{1 \leq m \leq M_\eta^*}\{\alpha_m y_1+\beta_m\}
    +
    \alpha_{\min}(y_2-y_1)\\
    &>
    \widehat{U}_{M_\eta^*}(y_1).
    \end{aligned}
    \]
    Hence, $\widehat{U}_{M_\eta^*}$ is strictly increasing. Moreover, for every
    $\varepsilon\geq\bar\varepsilon$, we have
    $\mathcal{B}_\varepsilon(\widehat{\mathbb{F}})=\mathcal M(\mathfrak{X})$.
    Therefore, the hyperplane-approximation surrogate problem~\ref{problem: supporting hyperplane DRO surrogate} reduces to
    \[
    \max_{w \in \Delta_n} \, \min_{x \in \mathfrak{X}} \widehat{U}_{M_\eta^*} (\langle w, x\rangle).
    \]
    Because $\widehat{U}_{M_\eta^*}$ is strictly increasing,
    \[
    \arg\max_{w \in \Delta_n}\min_{x\in\mathfrak{X}} \widehat{U}_{M_\eta^*}(\langle w,x\rangle)
    =
    \arg\max_{w\in\Delta_n}\min_{x\in\mathfrak{X}}\langle w,x\rangle.
    \]
    Since $w\geq \mathbf 0$ and
    $
    \mathfrak{X}=\mathfrak{X}_{\rm box},
    $
    we have
    \[
    \min_{x\in\mathfrak{X}}\langle w,x\rangle
    =
    \langle w,x_{\min}\rangle.
    \]
    The result now follows directly from Theorem~\ref{theorem: large ambiguity optimality}.
\end{proof}

\medskip
\section{Constraint-Generation Implementation of the Exact Benchmark}
\label{appendix: exact constraint generation}

This appendix describes the implementation of $\texttt{DRO}_{\texttt{EX-CG}}$ used in Section~\ref{sec: Computational Scalability}.
The method specializes the standard constraint-generation principle \cite{hettich1993semi} to the exact sample-specific vertex reformulation in Proposition~\ref{prop: exact sample vertex reformulation}. 
At each iteration, it solves a restricted master problem over the currently active sample--vertex constraints and then, for each sample $j$, solves the separation problem over $\mathcal V_j$ to identify violated constraints. The procedure terminates when no violation exceeds the prescribed numerical~tolerance.

\begin{lemma}[Exact Prefix-Vertex Reduction]
\label{lemma: prefix separation}
Under the conditions of
Proposition~\ref{prop: exact sample vertex reformulation}, fix
$w\in\mathcal W$, $\lambda\geq0$, and
$\widehat x_j\in\mathfrak{X}$. Let $\pi:=(\pi_1,\dots,\pi_n)$ be a permutation of $\{1,\dots,n\}$ satisfying
$
    w_{\pi_1}\geq w_{\pi_2}\geq\cdots\geq w_{\pi_n},
$
and define the nested prefix sets
\[
    S_k:=\{\pi_1,\dots,\pi_k\},
    \qquad k=0,\dots,n,
\]
where $S_0=\varnothing$. For each $k$, let $v_j^k\in\mathcal V_j$
be the vertex whose $i$th coordinate equals $(x_{\min})_i$ if
$i\in S_k$ and $(\widehat x_j)_i$ otherwise. Then
\begin{align*}
    \min_{v\in\mathcal V_j}
    \left\{
        U(\langle w,v\rangle) +\lambda\|v-\widehat x_j\|_1
    \right\} 
     =
    \min_{0\leq k\leq n}
    \left\{
        U(\langle w,v_j^k\rangle) +\lambda\|v_j^k-\widehat x_j\|_1
    \right\}.
\end{align*}
\end{lemma}

\begin{proof}[Proof of Lemma~\ref{lemma: prefix separation}]
    Fix the sample~$j$ and define
$
    d_{ji}:=(\widehat x_j)_i-(x_{\min})_i\geq0.
$
For every vertex $v\in\mathcal V_j$, there exists a subset
$S\subseteq\{1,\dots,n\}$ such that
\[
v_i=
\begin{cases}
	(x_{\min})_i, & i\in S,\\
	(\widehat x_j)_i, & i\notin S.
\end{cases}
\]
Define its transport distance $D_j$ and portfolio return reduction $L_j$ by
\[
    D_j(S):=\sum_{i\in S}d_{ji},
    \qquad
    L_j(S):=\sum_{i\in S}w_i d_{ji}.
\]
Then, for the vertex $v$ represented by $S$, it is readily verified that
\[
\langle w,v\rangle=\langle w,\widehat x_j\rangle - L_j(S),
\qquad
\|v-\widehat x_j\|_1 = D_j(S).
\]
Consequently, the minimization over $\mathcal V_j$ can be
written equivalently as
\begin{align}
    &\min_{v\in\mathcal V_j}
    \left\{
        U(\langle w,v\rangle)
        +\lambda\|v-\widehat x_j\|_1
    \right\}
   =
    \min_{S\subseteq\{1,\dots,n\}}
    \left\{
        U\bigl(\langle w,\widehat x_j\rangle-L_j(S)\bigr)
        +\lambda D_j(S)
    \right\}.
    \label{eq: subset separation problem}
\end{align}

To analyze the minimization on the right-hand side, fix a total
transportation distance
$
    D\in\left[0, \; \sum_{i=1}^n d_{ji}\right].
$
Among subsets satisfying $D_j(S)=D$, the transportation term
$\lambda D$ is constant. Since $U$ is increasing, minimizing~\eqref{eq: subset separation problem} at this fixed distance amounts to maximizing the
corresponding return reduction $L_j(S)$. We therefore introduce continuous variables $z_i\in[0,d_{ji}]$,
where $z_i$ denotes the transportation distance assigned to
coordinate~$i$, and consider
\[
L_j^\star(D)
:=
\max\left\{
\sum_{i=1}^n w_i z_i:
\sum_{i=1}^n z_i=D,\;
0\leq z_i\leq d_{ji}
\right\},
\]
which is a fractional knapsack problem.
Suppose that a feasible $z$ admits indices $i,\ell \in \{1,\dots,n\}$ such that
\[
w_i>w_\ell,\qquad z_i<d_{ji},\qquad z_\ell>0.
\]
For any
$
0<\xi\leq\min\{d_{ji}-z_i,z_\ell\},
$
transferring $\xi$ from $z_\ell$ to $z_i$ preserves constraint
$
\sum_{r=1}^n z_r=D
$
and increases the objective by $\xi(w_i-w_\ell)>0$. Hence, an
optimizer can be chosen to fill the coordinates in the order~$\pi$, with ties broken
arbitrarily.
Thus, for every feasible $D$, there exist
$k\in\{0,\dots,n-1\}$ and $\theta\in[0,1]$ such that an optimizer
satisfies
\[
z_{\pi_r}=d_{j\pi_r}\quad(r=1,\dots,k),\qquad
z_{\pi_{k+1}}=\theta d_{j\pi_{k+1}},
\]
and $z_{\pi_r}=0$ for $r\geq k+2$. Consequently,
\[
\begin{aligned}
D
&=(1-\theta)D_j(S_k)+\theta D_j(S_{k+1}),\\
L_j^\star(D)
&=(1-\theta)L_j(S_k)+\theta L_j(S_{k+1}).
\end{aligned}
\]

The graph of $L_j^\star$ therefore consists of line
segments joining the consecutive prefix points
\[
\bigl(D_j(S_k),L_j(S_k)\bigr) = \Bigl(\sum_{r=1}^k d_{j \pi_r}, \; \sum_{r=1}^k w_{\pi_r} d_{j \pi_r} \Bigr),
\qquad k=0,\dots,n
\]
with $(D_j(S_0), L_j(S_0))=(0,0).$
Now fix any subset~$S$. The choice
$
z_i=d_{ji}\mathbf 1_{\{i\in S\}}
$
is feasible at $D=D_j(S)$ and attains $L_j(S)$. Thus,
\[
L_j^\star(D_j(S))\geq L_j(S).
\]
Since $U$ is increasing, it follows that
\[
U\bigl(\langle w,\widehat x_j\rangle-L_j^\star(D_j(S))\bigr)
+\lambda D_j(S)
\leq
U\bigl(\langle w,\widehat x_j\rangle-L_j(S)\bigr)
+\lambda D_j(S).
\]

The point
$
\bigl(D_j(S),L_j^\star(D_j(S))\bigr)
$
lies on a segment joining two consecutive prefix points. Along this
segment, $D$ and $L_j^\star(D)$ vary affinely. Since $U$ is concave,
$
U\bigl(\langle w,\widehat x_j\rangle-L_j^\star(D)\bigr)+\lambda D
$
is concave in the segment parameter and therefore attains its minimum
at an endpoint. Hence, for every subset~$S$, some prefix set~$S_k$
has an objective value no greater than that of~$S$.

Thus, the minimum over the prefix sets is no greater than the minimum
over all subsets. The reverse inequality follows because every prefix
set is itself a feasible subset.
\end{proof}

\medskip
\section{Transaction-Cost Sensitivity of the Rolling Experiment}
\label{appendix: transaction cost sensitivity}

This appendix supplements the fixed 476-asset rolling experiment in
Section~\ref{sec: Multiple-Period Case} above. We
apply proportional transaction-cost rates of $0.1\%$, $0.2\%$, and
$0.3\%$ to risky-asset turnover at each monthly rebalancing date.
The costs are assessed ex post and are not included in the
portfolio-optimization objective. The buy-and-hold benchmarks incur
no subsequent rebalancing turnover.

Figure~\ref{figure: supplement transaction cost trajectories} reports
the resulting account-value trajectories, and
Table~\ref{table: supplement transaction cost sensitivity} gives the
associated performance metrics. Increasing transaction costs reduces cumulative returns, Sharpe ratios, and Calmar ratios across the rebalanced portfolios, while volatility and maximum drawdown remain comparatively stable.
The effect is most pronounced for small values of~$\varepsilon$, whose
allocations rebalance more aggressively. The portfolio with
$\varepsilon=10^{-2}$ retains the highest Sharpe and Calmar ratios among the tested
rebalanced policies at every cost level. At $\varepsilon=1$, the
allocation is already close to equal weight, but periodic rebalancing
still incurs costs from weight drift and therefore underperforms the
equal-weight buy-and-hold benchmark.

\begin{figure}[htbp]
    \centering
    \includegraphics[width=.7\linewidth]{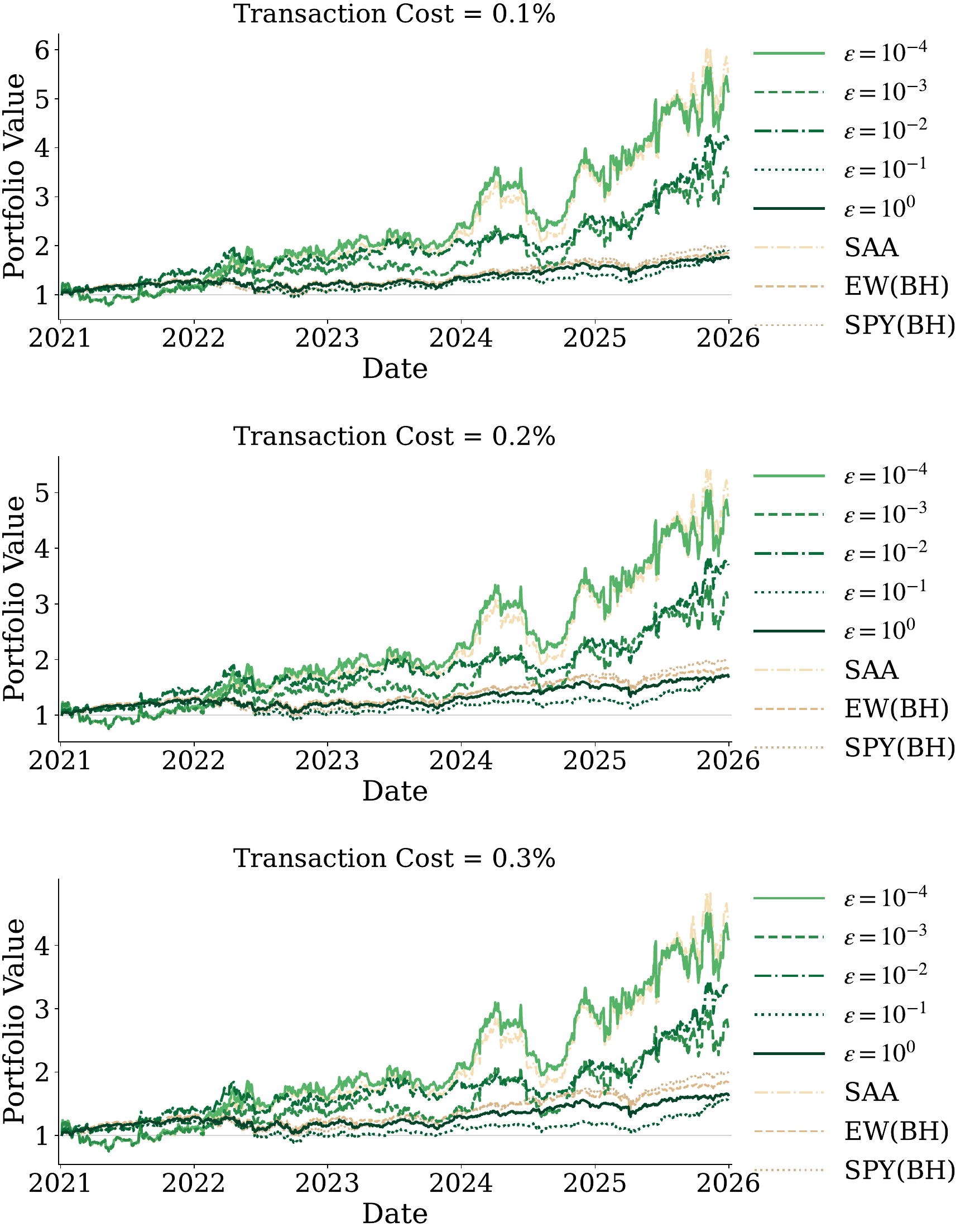}
    \caption{Out-of-sample account trajectories for the fixed 476-asset rolling portfolio under transaction-cost rates of $0.1\%$, $0.2\%$, and $0.3\%$ over 2021--2025.}
    \label{figure: supplement transaction cost trajectories}
\end{figure}

\begin{table}[htbp]
\footnotesize
\centering
\caption{Transaction-cost sensitivity of the rebalanced portfolios over 2021--2025.}
\label{table: supplement transaction cost sensitivity}
\renewcommand{\arraystretch}{1.08}
\begin{tabular}{@{}llccccc@{}}
\toprule
TC & Portfolio & CR & $\sigma$ & SR & MDD & Calmar \\
\midrule
0.1\%
& \texttt{SAA}          & 4.52 & 0.46 & 0.90 & 0.39 & 1.05 \\
& $\varepsilon=10^{-4}$ & 4.15 & 0.45 & 0.88 & 0.38 & 1.02 \\
& $\varepsilon=10^{-3}$ & 2.40 & 0.44 & 0.71 & 0.39 & 0.71 \\
& $\varepsilon=10^{-2}$ & 3.14 & 0.29 & 1.03 & 0.27 & 1.21 \\
& $\varepsilon=10^{-1}$ & 0.89 & 0.19 & 0.60 & 0.29 & 0.47 \\
& $\varepsilon=10^{0}$  & 0.75 & 0.16 & 0.59 & 0.20 & 0.59 \\
\midrule
0.2\%
& \texttt{SAA}          & 3.93 & 0.46 & 0.85 & 0.39 & 0.96 \\
& $\varepsilon=10^{-4}$ & 3.60 & 0.45 & 0.83 & 0.39 & 0.92 \\
& $\varepsilon=10^{-3}$ & 2.03 & 0.44 & 0.66 & 0.39 & 0.63 \\
& $\varepsilon=10^{-2}$ & 2.70 & 0.29 & 0.95 & 0.27 & 1.09 \\
& $\varepsilon=10^{-1}$ & 0.71 & 0.19 & 0.49 & 0.30 & 0.38 \\
& $\varepsilon=10^{0}$  & 0.70 & 0.16 & 0.54 & 0.20 & 0.55 \\
\midrule
0.3\%
& \texttt{SAA}          & 3.40 & 0.46 & 0.80 & 0.40 & 0.87 \\
& $\varepsilon=10^{-4}$ & 3.10 & 0.45 & 0.78 & 0.39 & 0.83 \\
& $\varepsilon=10^{-3}$ & 1.70 & 0.44 & 0.60 & 0.40 & 0.55 \\
& $\varepsilon=10^{-2}$ & 2.30 & 0.29 & 0.87 & 0.28 & 0.98 \\
& $\varepsilon=10^{-1}$ & 0.54 & 0.19 & 0.38 & 0.30 & 0.30 \\
& $\varepsilon=10^{0}$  & 0.64 & 0.16 & 0.50 & 0.21 & 0.50 \\
\midrule
\multicolumn{7}{@{}l}{\emph{Buy-and-hold benchmarks}}\\
-- & \texttt{EW(BH)}  & 0.83 & 0.16 & 0.63 & 0.20 & 0.66 \\
-- & \texttt{SPY(BH)} & 0.98 & 0.17 & 0.70 & 0.25 & 0.60 \\
\bottomrule
\end{tabular}

\vspace{0.3em}
\parbox{0.95\columnwidth}{\footnotesize
\emph{Note:} The no-transaction-cost results are reported in
Table~\ref{table: Mult Performance Metrics of SP500} above. The \texttt{SAA} row corresponds to the $\varepsilon=0$
endpoint. Transaction costs are applied only to turnover from
rebalanced risky-asset positions.
}
\end{table}

\clearpage


\small

\setlength{\bibsep}{5pt}
\bibliographystyle{apalike}
\bibliography{references}

\end{document}